\documentclass[11pt]{article}
\usepackage[tbtags]{amsmath}
\usepackage{cases}
\usepackage{mathrsfs}
\usepackage{amsfonts}
\usepackage{amsmath,amsthm,amssymb}

\usepackage{color}

\newcommand{\ds}{\displaystyle}

\theoremstyle{plain}
\newtheorem{theorem}{Theorem}[section]
\newtheorem{lemma}{Lemma}[section]
\newtheorem{remark}{Remark}[section]
\newtheorem{proposition}{Proposition}[section]
\newtheorem{corollary}{Corollary}[section]
\newif \ifLastSection \LastSectionfalse

\numberwithin{equation}{section}

\begin{document}

\title{{\bf {Convergence to diffusion waves for solutions of Euler equations with time-depending damping on quadrant}}}

\author{H{\sc aibo}  C{\sc ui}\thanks{School of Mathematical Sciences, Huaqiao University, Quanzhou 362021, P.R. China. Email:
hbcui@hqu.edu.cn}, \quad H{\sc aiyan} Y{\sc in}\thanks{School of Mathematical Sciences, Huaqiao University, Quanzhou 362021, P.R. China. Email: hyyin@hqu.edu.cn},  \quad C{\sc hangjiang} Z{\sc
hu}\thanks{Corresponding author. School of Mathematics, South China
University of Technology, Guangzhou 510641, P.R. China. Email:
machjzhu@scut.edu.cn},
\quad L{\sc imei} Z{\sc
hu}\thanks{School of Mathematics, South China
University of Technology, Guangzhou 510641, P.R. China. Email:
287365401@qq.com} }

\date{}

\maketitle

\textbf{{\bf Abstract:}} This paper is concerned with the asymptotic
behavior of  the solution to the Euler equations with time-depending
damping on quadrant $(x,t)\in \mathbb{R}^+\times\mathbb{R}^+$,
\begin{equation}\notag
\partial_t  v
  -
   \partial_x u=0, \qquad
\partial_t u
   +
      \partial_x  p(v)
     =\displaystyle
     -\frac{\alpha}{(1+t)^\lambda} u,
\end{equation}
with null-Dirichlet boundary condition  or null-Neumann boundary
condition on $u$. We show that the corresponding initial-boundary value
problem admits a unique global smooth solution which tends
time-asymptotically  to the nonlinear diffusion wave.  Compared with
the previous work about
 Euler equations with constant coefficient damping,  studied by Nishihara and Yang  (1999, J. Differential Equations, 156,
439-458), and Jiang and Zhu (2009, Discrete Contin. Dyn. Syst., 23, 887-918), we obtain a general result  when the initial
perturbation belongs to the same space. In  addition, our main novelty lies in the facts that the cut-off points of the convergence rates  are different from our  previous result about the Cauchy problem. Our proof is based on the classical energy method and the  analyses  of the nonlinear diffusion wave.

\bigbreak \textbf{{\bf Key Words}:}  Euler equations with
time-depending damping, nonlinear diffusion waves, initial-boundary
value problem, decay estimates.

\bigbreak  {\textbf{AMS Subject Classification :} 35L65, 76N15,
35B45, 35B40.}

\tableofcontents
\section{Introduction}
In this paper, we consider the asymptotic behavior and the convergence rates of solutions to the one-dimensional compressible Euler equations with time-depending damping:
\begin{equation}\label{1.3}
\left\{\begin{array}{l}
\partial_t  v
  -
   \partial_x u=0,\\[2mm]
\partial_t u
   +
      \partial_x  p(v)
     =\displaystyle
     -\frac{\alpha}{(1+t)^\lambda} u,\qquad (x,t)\in \mathbb{R}^+\times\mathbb{R}^+,
 \end{array}
        \right.
\end{equation}
with initial data
\begin{equation}\label{1.4}
  (v,u)\mid_{t=0}=(v_0,u_0)(x)\rightarrow (v_+,u_+),
   \quad
    \mbox{as}
     \quad
     x\rightarrow  +\infty
      \quad
        \mbox{and}
         \quad
           v_+>0,
\end{equation}
and  the null-Dirichlet boundary condition
\begin{equation}\label{bound1}
u(0,t)=0,
\end{equation}
or the null-Neumann boundary condition
\begin{equation}\label{bound2}
\partial_x u(0,t)=0,
\end{equation}
where  $v=v(x,t)>0$ is the specific volume, $u=u(x,t)$ is the velocity  and the pressure $p(v)>0$ is a smooth function with
$p'(v)<0$ for any $v\in \mathbb{R}^+$. The external term
$-\frac{\alpha}{(1+t)^\lambda} u$ with  physical coefficients
$\alpha>0$ and $\lambda\geq0$, is called a time-depending
damping. $v_+>0$ and $u_+$ are constant states.

The system  \eqref{1.3} is not only a mathematical model of the  wave equation   with time-depending dissipation \cite{Wirth1,Wirth2}
 \begin{equation}\notag
\omega_{tt}-\omega_{xx}+b(t)\omega_t=0,
\end{equation}
 but it models the compressible flow through porous media with unsteady drag force. For more information about this model, see for instance \cite{Dafermos} and
references cited therein for related models.

When $\alpha=0$, the system \eqref{1.3} reduces to the standard
compressible Euler equations which is an extremely important
equation to describe the motion of compressible ideal fluids. There
have been many important
 developments and extensive studies on the Euler equations in the past
few decades.

%(see, for example \cite{Alinhac1999-1,Alinhac1999-2} and references therein)
%It is well known that smooth solutions of \eqref{1.3} will in general blow up in finite time.
%
When $\alpha>0, \lambda=0$, the system \eqref{1.3} becomes the
compressible Euler equations with constant coefficient damping.
The global
   existence and large time behaviors of smooth solutions to the Cauchy problem or initial-boundary problem \eqref{1.3} have been investigated by many authors.

 (i) For the Cauchy problem, the global existence of solution has been investigated by many authors (see \cite{Nishida1978} and references therein). Hsiao and Liu \cite{Hsiao-Liu1992} firstly considered  the large time behavior of solution.
 Precisely, they showed that the solution of \eqref{1.3} tended time-asymptotically to
 the nonlinear diffusion waves. And a better convergence rate was obtained by Nishihara \cite{Nishihara1996}. In the case of large initial data, Zhao in \cite{Zhao2001} showed that for a certain class of given large initial data, the system \eqref{1.3} admits a unique global smooth solution  and such a solution tends time-asymptotically to the strong diffusion wave. For other results, see \cite{Huang-Pan2003,Huang-Pan2006,Nishihara-Wang-Yang2000,Wang-Yang2003,Zhu2003} and some references therein.

(ii) For the initial-boundary value problem on a half line $\mathbb{R}^+$, we
refer to
\cite{Jiang-Zhu2009,Marcati-Mei2000,Marcati-Mei-Rubino2005,Nishihara-Yang1999}. Precisely, Marcati and Mei in \cite{Marcati-Mei2000} studied the system
\eqref{1.3} with boundary condition on $v$ as follows:
\begin{equation}\notag
v(0,t)=g(t), \qquad t>0.
\end{equation}
 Nishihara and Yang in  \cite{Nishihara-Yang1999} considered the asymptotic behavior of solution to the system \eqref{1.3} with the Dirichlet boundary condition \eqref{bound1} or the Neumann boundary condition \eqref{bound2}. More precisely,  for the Dirichlet boundary condition,  they got the global  existence and convergence rates in form of  $\|(v-\bar{\bar v},u-\bar{\bar u})\|_{L^\infty}\leq C(t^{-\frac{3}{4}}, t^{-\frac{5}{4}})$  by perturbing the initial value around  the linear diffusion waves $(\bar{\bar{v}}, \bar{\bar{u}})(x,t)$ which
 satisfied
  \begin{equation}\notag
\left\{\begin{array}{l}
\partial_t  \bar{\bar{v}}
  -
   \partial_x \bar{\bar{ u}}=0,\\[2mm]
       p'(v_+)\partial_x\bar{\bar{v}}
     =\displaystyle
     -\alpha \bar{\bar{u}},\qquad (x,t)\in \mathbb{R}^+\times\mathbb{R}^+,\\[2mm]
     (\bar{\bar v},\bar{\bar u})\mid_{t=0}=(\bar{\bar v}_0,\bar{\bar u}_0)(x)\rightarrow (v_+,0),\quad
    \mbox{as}
     \quad
     x\rightarrow  +\infty,\\[2mm]
     \bar{\bar u}(0,t)=0 \ \ (\mbox{or} \ \ \partial_x\bar{\bar{v}}(0,t)=0  ),\qquad \bar{\bar{u}}(
     \infty,t)=0,
 \end{array}
        \right.
\end{equation}
provided the initial perturbation belonged to $H^3(\mathbb{R}^+)\times H^2(\mathbb{R}^+)$.
 In \cite{Marcati-Mei-Rubino2005}, for the Dirichlet boundary condition,  Marcati, Mei and Rubino improved the convergence rates to  $\|(v-\check v,u-\check u)\|_{L^\infty}\leq C(t^{-1}, t^{-\frac{3}{2}})$  by perturbing the initial value around  the nonlinear diffusion waves $(\check{v}, \check{u})(x,t)$ which
 satisfied
 \begin{equation}\notag
\left\{\begin{array}{l}
\partial_t  \check{v}
  -
   \partial_x \check{ u}=0,\\[2mm]
      \partial_x  p(\check{v})
     =\displaystyle
     - \alpha  \check{u},\qquad (x,t)\in \mathbb{R}^+\times\mathbb{R}^+,\\[2mm]
     (\check v,\check u)\mid_{t=0}=(\check v_0,\check u_0)(x)\rightarrow (v_+,0),\quad
    \mbox{as}
     \quad
     x\rightarrow  +\infty,\\[2mm]
     \check u(0,t)=0 \ \ (\mbox{or} \ \ \partial_x \check{ v}(0,t)=0  ),\qquad \check{u}(  \infty,t)=0,
 \end{array}
        \right.
\end{equation}
  when the initial perturbation additionally belonged to $ L^1(\mathbb{R}^+)$.
Later,  Jiang and Zhu in \cite{Jiang-Zhu2009}
obtained the same convergence rates  as in
\cite{Marcati-Mei-Rubino2005} under a rather weaker small assumption
on the initial disturbance. For the initial-boundary value problem on
the bounded domain $[0,1]$, we refer to \cite{Hsiao-Pan1999}. For
initial-boundary value problem to the compressible Euler equations
with nonlinear damping, we refer to
\cite{Jiang-Zhu20092,Lin-Lin-Mei2010} and some references therein.

When  $\alpha>0, \lambda>0$,  the system \eqref{1.3} is the
compressible Euler equations with time-depending damping.  The
authors   \cite{Hou-Witt-Yin2015,Hou-Yin2016}  considered the
global existence of smooth solutions when $u_0\in
C_0^\infty(\mathbb{R}^3)$ for $0\leq \lambda \leq 1$ in
multi-dimensions. And they proved the solutions will blow up in
finite time for $\lambda>1$. For more results about this direction,
we can refer to \cite{Pan2016,Pan2017} and references cited therein.
Recently, the authors   \cite{Cui-Yin-Zhang-Zhu2016} considered
the Cauchy problem for the system \eqref{1.3} and proved that the
solution time-asymptotically converged to the nonlinear diffusion
waves  with the initial perturbation
around the diffusion wave in  $H^3(\mathbb{R})\times
H^2(\mathbb{R})$. For more  literature on this model, see \cite{Hou2017,Pan3} and the references therein.

However, to our knowledge, there are very few results  on the
large-time behavior  of solutions near nonlinear diffusion waves to
the initial-boundary value problem \eqref{1.3}-\eqref{bound2} for $0<\lambda<1$. It
is very interesting and challenging to study this problem because it has more
physical meanings and of course some new mathematical difficulties
 will arise due to the boundary effect and time-depending damping.
 In this paper, we will consider the initial-boundary value problem of \eqref{1.3} on a half line $\mathbb{R}^+ $
 and obtain the convergence to diffusion waves for classical solution  compared with previous results about
 Euler equations with constant coefficient damping.

%Before concluding this section,

Now, we give the main ideas used in
deducing our results.   In the case of  Dirichlet boundary condition, the main difficulty of this paper lies in obtaining  the decay rates of the diffusion waves. The general strategy is to construct the self-similar solution.   However, in our diffusion waves system, the equations  don't possess self-similar solution. One  possible way to get around these issues is to   explicitly
write  out the solution by Green function as in Nishihara and Yang \cite{Nishihara-Yang1999}. But, in our case, we cannot achieve the expected results because  the diffusion waves is nonlinear. Another option would be to construct nonlinear diffusion wave by iteration. However, achieving the Green function of the Dirichlet type IBVP to the nonlinear  diffusion   waves  share the same difficulty as directly solving the solution.   Our strategy,  inspired by the work of  Jiang and Zhu \cite{Jiang-Zhu20092}, is to employ extension the initial data to the real line and consider  the
corresponding linearized problem \eqref{a1.a7}. Based on  some
delicate energy estimates, we can get the decay rates of nonlinear
diffusion waves $(\bar v,\bar u)(x,t)$ indirectly. Because of the different decay rates of the diffusion waves, we obtain the cut-off point of the convergence rate is $\lambda=\frac{3}{5}$  while  the Cauchy problem
\cite{Cui-Yin-Zhang-Zhu2016} is $\lambda=\frac{1}{7}$. On the other hand,  for the case of Neumann boundary condition, we construct the self-similar diffuse waves. Therefore, the cut-off point of the convergence rate is $\lambda=\frac{1}{7}$ if $v_0(0)\neq v_+$.  However, there is no cut-off point of the convergence rate if $v_0(0)=v_+$ because the diffusion waves are constant states.  Finally, we also   take full use of the weight
function $(1+t)^\beta$  to
overcome  the time-depending damping.

The rest of the paper is organized as follows.  In Section
\ref{S2}, we derive the convergence in the case of Dirichlet boundary condition.  In Section
\ref{S2.1}, the problem with  null-Dirichlet boundary condition is
reformulated and the main results will be stated.
  In Section \ref{S2.2}, we will obtain the dissipative properties of the nonlinear diffusion waves $(\bar u, \bar u)(x,t)$. In Section \ref{S2.3}, the proofs of Theorem will be given, much of that is base on the papers \cite{Cui-Yin-Zhang-Zhu2016}. In Section  \ref{S4}, we will study the null-Neumann boundary problem.

{\bf Notations}: In the following, $C$ and  $c$($C_i$,$c_i$) denote the generic positive constants depending only on the initial data
and the physical coefficients $\alpha,\lambda$, but independent of the time.  For two quantities $a$ and $b$, $a \sim b$ means $\frac{1}{C}|b|\leq |a| \leq C|b|$
for a generic constant $C$.   $\|\cdot\|_{L^p}$ and $\|\cdot\|_{H^l}$ are denote by $\|\cdot\|_{L^p(\mathbb{R}^+)}$,  $\|\cdot\|_{H^l(\mathbb{R}^+)}$, for $1\leq p\leq \infty$, $l\geq 0$, respectively.

\section{The case of Dirichlet boundary condition}\label{S2}
{\numberwithin{equation}{subsection}

\subsection{Reformulation of the problem and main results}\label{S2.1}

We firstly consider the problem \eqref{1.3}-\eqref{1.4} with the
Dirichlet boundary condition \eqref{bound1}. Hinted by
$(\ref{1.3})_2$ and the initial data \eqref{1.4}, we suppose for any
$t\geq 0$
\begin{equation}\notag
 u(x,t)\rightarrow u_+\beta(t),
   \quad
    \mbox{as}
     \quad
     x\rightarrow +  \infty ,
\end{equation}
where
\begin{equation}\notag
\beta(t)=
\left\{
\begin{array}{l}
e^{-\frac{\alpha}{1-\lambda}[(1+t)^{1-\lambda}-1]},\qquad\   \mbox{ if} \qquad \lambda\in[0,1),\\
(1+t)^{-\alpha},\qquad\qquad\qquad \mbox{if} \qquad \lambda=1.
\end{array}
\right.
\end{equation}
Denote
\begin{equation}\notag
B(t)=-\int_t^\infty \beta(\tau)d \tau.
\end{equation}
From Darcy' law and asymptotic analysis, it is well-known that the first term $u_t$ of $(\ref{1.3})_2$ decay to zero, as $t\rightarrow  \infty$, faster than the term $-\frac{\alpha}{(1+t)^\lambda} u$ for some  $0\leq\lambda<1$. Therefore, we expect the solution $(v,u)(x,t)$  of
\eqref{1.3}-\eqref{bound1} time-asymptotically behaves as the solutions $(\bar v,\bar u)(x,t)$ of
\begin{equation}\label{a1.5}
\left\{\begin{array}{l}
\partial_t  \bar{v}
  -
   \partial_x \bar{ u}=0,\\[2mm]
      \partial_x  p(\bar{v})
     =\displaystyle
     -\frac{\alpha}{(1+t)^\lambda} \bar{u},\qquad (x,t)\in \mathbb{R}^+\times\mathbb{R}^+,\\[2mm]
     (\bar v,\bar u)\mid_{t=0}=(\bar v_0,\bar u_0)(x)\rightarrow (v_+,0),\quad
    \mbox{as}
     \quad
     x\rightarrow  +\infty,\\[2mm]
     \bar u(0,t)=0 \ \ (\mbox{or} \ \ \bar v_x(0,t)=0  ),\qquad \bar{u}(  \infty,t)=0,
 \end{array}
        \right.
\end{equation}
where $\bar v_0(x)$ satisfies
\begin{equation}\notag
\bar v_0(x)>0,
\end{equation}
and
\begin{equation}\label{a1.61}
\begin{split}
\int_0^\infty(\bar v_0(y)-v_+)dy
 =\int_0^\infty( v_0(y)-v_+)dy-u_+ B(0)
.
\end{split}
\end{equation}
Therefore
\begin{equation}\label{a1.6}
\int_0^\infty(v_0(y)-\bar v_0(y))dy=u_+ B(0).
\end{equation}

Next, as in \cite{Nishihara-Yang1999}, we define a pair of correction functions
\begin{equation}\label{1.10}
\hat{v}(x,t)=u_+ m_0(x)B(t)
\end{equation}
and
\begin{equation}\label{1.11}
\hat{u}(x,t)=u_+\beta(t)\int_{0}^x m_0(y)dy,
\end{equation}
where $m_0(x)$ is a smooth function with compact support such that
\begin{equation}\notag
\int_{0}^\infty m_0(x)dx=1, \qquad \mbox{supp}\  m_0(x)\subset \mathbb{R}^+.
\end{equation}
Therefore, $(\hat v,\hat u)(x,t)$ satisfies
\begin{equation}\label{1.12}
\left\{\begin{array}{l}
\partial_t  \hat{v}
  -
   \partial_x \hat{ u}=0,\\[2mm]
      \partial_t\hat{u}
     =\displaystyle
     -\frac{\alpha}{(1+t)^\lambda}\hat{u},\qquad (x,t)\in \mathbb{R}^+\times\mathbb{R}^+,\\[2mm]
     \hat u(0,t)=0,\qquad (\hat v, \hat u)(\infty,t)=(0,u_+\beta(t)).
 \end{array}
        \right.
\end{equation}
Combining \eqref{1.3}, \eqref{a1.5} and \eqref{1.12}, we have
\begin{equation}\label{1.17}
\left\{\begin{array}{l}
\partial_t  (v-\bar{v}-\hat{v})
  -
   \partial_x (u-\bar{ u}-\hat{u})=0,\\[2mm]
      \partial_t (u-\bar{u}-\hat{u})+\partial_x  (p(v)-p(\bar{v}))
         + \partial_t \bar u
           +\displaystyle
              \frac{\alpha}{(1+t)^\lambda} (u-\bar{u}-\hat{u})=0\\[2mm]
               (u-\bar{u}-\hat{u})(0,t)=0,\qquad (u-\bar{u}-\hat{u})(\infty,t)=0.
 \end{array}
        \right.
\end{equation}
By \eqref{a1.6} and $(\ref{1.17})_3$, the integration of $(\ref{1.17})_1$ with respect to $x$  and $t$ over $\mathbb{R}^+\times [0,t]$ yields
\begin{equation}\notag
\int_{0}^\infty(v -\bar v -\hat v )dx=\int_0^\infty (v_0(x)-\bar v_0(x))dx- u_+ B(0)=0,
\end{equation}
and hence we reach the setting of perturbation
\begin{equation}\label{1.13}
\omega(x,t)=-\int_x^{\infty} (v(y,t)-\bar v(y,t)-\hat{v}(y,t))dy,
\end{equation}
and
\begin{equation}\label{1.18}
z(x,t)= u(x,t)-\bar u(x,t)-\hat{u}(x,t).
\end{equation}
By \eqref{1.17}, we have the reformulated problem
\begin{equation}\label{1.19}
\left\{\begin{array}{l}
  \omega_t
  -
   z=0,\qquad\qquad\qquad (x,t)\in \mathbb{R}^+\times\mathbb{R}^+,\\[2mm]
       z_t+ (p(\omega_x+\bar v+\hat v)-p(\bar{v}))_x
           +\displaystyle
              \frac{\alpha}{(1+t)^\lambda} z=-\bar u_t,\\[2mm]
              \omega(0,t)=0, \qquad z(0,t)=0.
 \end{array}
        \right.
\end{equation}
with initial data
\begin{equation}\label{1.19idata}
  (\omega,z)\mid_{t=0}=(\omega_0,z_0)(x),
\end{equation}
where
\begin{equation}\notag
\left\{\begin{array}{l}
  \ds \omega_0(x)
    =-\int_x^{\infty}  (v_0(y)-\bar v_0(y)-\hat{v}(y,0))dy,\\
   \ds z_0(x)
       =u_0(x)-\bar u_0(x)-\hat{u}(x,0).
\end{array}
\right.
\end{equation}
Rewrite  \eqref{1.19} and \eqref{1.19idata} as
\begin{equation}\label{1.20}
\left\{\begin{array}{l}
       \omega_{tt}+ (p'(\bar{v})\omega_x)_x
           +\displaystyle
              \frac{\alpha}{(1+t)^\lambda} \omega_t=F,\qquad (x,t)\in \mathbb{R}^+\times\mathbb{R}^+,\\[2mm]
              \omega(0,t)=0, \qquad \omega_t(0,t)=0,
 \end{array}
        \right.
\end{equation}
with initial data
\begin{equation}\label{k1.20}
  (\omega,\omega_t)\mid_{t=0}=(\omega_0,z_0)(x),
\end{equation}
where
\begin{equation}\label{1.21}
  F
   =
    \frac{1}{\alpha}(1+t)^\lambda p(\bar v)_{xt}
      +
       \frac{\lambda}{\alpha} { (1+t)^{\lambda-1} } { p(\bar v)_x}
          -
            (p(\omega_x+\bar v+\hat v)-p(\bar{v})-p'(\bar v)\omega_x)_x.
\end{equation}

%%%%%%%%%%%%%%%%%%%%%%%%%%%%%%%%%%%%%%%%%%%%%%%%%%%%%%%%%%%%%%%%%%%%%%%%%%%%%%%%%%%%%%%%%%%%%% 主要定理

\begin{theorem}\label{Thm 1} {\bf (Dirichlet boundary for $0\leq \lambda<\frac{3}{5}$).}
For $\alpha>0$, suppose that $v_0(x)-v_+\in L^1(\mathbb{R}^+)$, if we assume further that
both $\delta=\|  v_0-v_+\|_{L^1(\mathbb{R}^+)}+|u_+|+
\|V_0\|_{H^5(\mathbb{R})}+\|Z_0\|_{H^4(\mathbb{R})}$ and $\|\omega_0\|_{H^3(\mathbb{R}^+)}+\|z_0\|_{H^2(\mathbb{R}^+)}$ are sufficiently small.  Then,
there exists a unique time-global solution of the initial-boundary value problem \eqref{1.20}-\eqref{k1.20} satisfying
\begin{equation}\notag
  \omega\in C^{k}((0,\infty),H^{3-k}(\mathbb{R}^+)),\qquad k=0,1,2,3,
\end{equation}
\begin{equation}\notag
  \omega_t\in C^{k}((0,\infty),H^{2-k}(\mathbb{R}^+)),\qquad k=0,1,2,
\end{equation}
furthermore, we have
\begin{equation}\notag
\begin{split}
&
      \sum_{k=0}^3(1+t)^{(\lambda +1)k }\|\partial_x^k\omega(\cdot,t)\|_{L^2}^2
      +
         \sum_{k=0}^2(1+t)^{(\lambda+1)k +2} \|\partial_x^k\omega_t(\cdot,t)\|_{L^2}^2
          \\
         &
            +
                \int_0^t \bigg[ \sum_{j=1}^3(1+s)^{(\lambda +1)j-1}\|\partial_x^j\omega(\cdot,s)\|_{L^2}^2
                 +
                 \sum_{j=0}^2(1+s)^{(\lambda +1)j+1} \|\partial_x^j\omega_t(\cdot,s)\|_{L^2}^2\bigg]ds\\
                    \leq
                       &
                          C(\|\omega_0\|_{H^3(\mathbb{R}^+)}^2
                          +
                          \|z_0\|_{H^2(\mathbb{R}^+)}^2+\delta),
\end{split}
\end{equation}
where the initial error function $(V_0, Z_0)(x)$ will  be defined in \eqref{a1.19idata2}.
\end{theorem}

\begin{theorem}\label{Thm 1-0} {\bf (Dirichlet boundary for $\frac{3}{5}<\lambda<1$).}
For $\alpha>0$, suppose   $v_0(x)-v_+\in L^1(\mathbb{R}^+)$, if we assume further that
both $\delta=\|   v_0-v_+\|_{L^1(\mathbb{R}^+)}+|u_+|+
\|V_0\|_{H^5(\mathbb{R})}+\|Z_0\|_{H^4(\mathbb{R})}$ and $\|\omega_0\|_{H^3(\mathbb{R}^+)}+\|z_0\|_{H^2(\mathbb{R}^+)}$ are sufficiently small.  Then,
there exists a unique time-global solution of the initial-boundary value problem \eqref{1.20}-\eqref{k1.20} satisfying
\begin{equation}\notag
  \omega\in C^{k}((0,\infty),H^{3-k}(\mathbb{R}^+)),\qquad k=0,1,2,3,
\end{equation}
\begin{equation}\notag
  \omega_t\in C^{k}((0,\infty),H^{2-k}(\mathbb{R}^+)),\qquad k=0,1,2,
\end{equation}
furthermore, we have
\begin{equation}\notag
\begin{split}
     &
      \sum_{k=0}^3(1+t)^{(\lambda +1)k +\frac{3}{2}-\frac{5\lambda}{2}}
      \|\partial_x^k\omega(\cdot,t)\|_{L^2}^2
      +
         \sum_{k=0}^2(1+t)^{(\lambda+1)k +\frac{7}{2}-\frac{5\lambda}{2}} \|\partial_x^k\omega_t(\cdot,t)\|_{L^2}^2\\
                 \leq
                       &
                          C(\|\omega_0\|_{H^3}^2+\|z_0\|_{H^2}^2+\delta),
\end{split}
\end{equation}
and for any $\beta\in(\frac{3}{2}-\frac{3\lambda}{2},\lambda),$ we
have
\begin{equation}\notag
\begin{split}
         &
                \int_0^t \bigg[ \sum_{j=0}^3(1+s)^{(\lambda +1)(j-1)+\beta}\|\partial_x^j\omega(\cdot,s)\|_{L^2}^2
                 +
                 \sum_{j=0}^2(1+s)^{(\lambda +1)j+\beta-\lambda+1} \|\partial_x^j\omega_t(\cdot,s)\|_{L^2}^2\bigg]ds\\
                    \leq
                       &
                       C(1+t)^{\beta+\frac{3\lambda}{2}-\frac{3}{2}}
                           (\|\omega_0\|_{H^3}^2+\|z_0\|_{H^2}^2+\delta).
\end{split}
\end{equation}
\end{theorem}
\begin{theorem}\label{Thm 1-1} {\bf (Dirichlet boundary for $\lambda=\frac{3}{5}$).}
For $\alpha>0$, suppose   $v_0(x)-v_+\in L^1(\mathbb{R}^+)$, if we assume further that
both $\delta=\|   v_0-v_+\|_{L^1(\mathbb{R}^+)}+|u_+|+
\|V_0\|_{H^5(\mathbb{R})}+\|Z_0\|_{H^4(\mathbb{R})}$ and $\|\omega_0\|_{H^3(\mathbb{R}^+)}+\|z_0\|_{H^2(\mathbb{R}^+)}$ are sufficiently small.  Then,
there exists a unique time-global solution of the initial-boundary value problem \eqref{1.20}-\eqref{k1.20} satisfying
\begin{equation}\notag
  \omega\in C^{k}((0,\infty),H^{3-k}(\mathbb{R}^+)),\qquad k=0,1,2,3,
\end{equation}
\begin{equation}\notag
  \omega_t\in C^{k}((0,\infty),H^{2-k}(\mathbb{R}^+)),\qquad k=0,1,2,
\end{equation}
furthermore, for any sufficiently small $\varepsilon>0$ we have
\begin{equation}\notag
\begin{split}
     &
      \sum_{k=0}^3(1+t)^{\frac{8k}{5} }
      \|\partial_x^k\omega(\cdot,t)\|_{L^2}^2
      +
         \sum_{k=0}^2(1+t)^{\frac{8k}{5}+2} \|\partial_x^k\omega_t(\cdot,t)\|_{L^2}^2
          \\
         &
            +
                \int_0^t \bigg[ \sum_{j=1}^3(1+s)^{\frac{8j}{5}-1}\|\partial_x^j\omega(\cdot,s)\|_{L^2}^2
                 +
                 \sum_{j=0}^2(1+s)^{\frac{8j}{5}+1} \|\partial_x^j\omega_t(\cdot,s)\|_{L^2}^2\bigg]ds
         \\
                 \leq
                       &
                          C(1+t)^\varepsilon(\|\omega_0\|_{H^3}^2+\|z_0\|_{H^2}^2+\delta).
\end{split}
\end{equation}
\end{theorem}

Notice that $\omega_{x}=v-\bar v-\hat v$, $z=u-\bar u-\hat u$, and use
Proposition \ref{Prop a2} in the next subsection and Sobolev inequality, we
immediately obtain the following convergence rates.
\begin{corollary}\label{corollary 1}
Under the assumptions of Theorem \ref{Thm 1}-\ref{Thm 1-1},   the system
\eqref{1.3}-\eqref{bound1} possesses a uniquely global solution
$(v,u)(x,t)$ satisfying
\begin{equation}\notag
       \| (v-\bar v)(\cdot,t)\|_{L^\infty(\mathbb{R}^+)}
       \leq\left\{
\begin{array}{l}
    C(1+t)^{-\frac{3(\lambda+1)}{4}},\qquad\qquad \quad   0\leq\lambda<\frac{3}{5},\\[2mm]
    C(1+t)^{-\frac{ 6}{5}+\varepsilon},\qquad\qquad\quad\ \       \lambda=\frac{3}{5},\\[2mm]
     C(1+t)^{\frac{\lambda-3}{2}}, \qquad\qquad \quad  \quad    \frac{3}{5}<\lambda<1,
   \end{array}
   \right.
\end{equation}
and
  \begin{equation}\notag
       \| (u-\bar u)(\cdot,t)\|_{L^\infty(\mathbb{R}^+)}\leq\left\{
\begin{array}{l}
     C(1+t)^{-\frac{\lambda+5}{4}},\qquad\qquad \quad   0\leq\lambda<\frac{3}{5},\\[2mm]
    C(1+t)^{-\frac{ 7}{5}+\varepsilon},\qquad\qquad\quad\ \    \lambda=\frac{3}{5},\\[2mm]
     C(1+t)^{\lambda-{2}}, \qquad\qquad \quad \quad    \frac{3}{5}<\lambda<1.
   \end{array}
   \right.
\end{equation}
\end{corollary}
\begin{remark}\label{remark 1}
 It should be noted that   the cut-off point of the convergence rate in this paper is $\lambda=\frac{3}{5}$, while  the Cauchy problem  in \cite{Cui-Yin-Zhang-Zhu2016} is $\lambda=\frac{1}{7}$. This is caused by the  diffusion wave constructed in this paper is not be self-similar solution. It is worth pointing out that the time-depending damping could reveal more phenomena  about the ware equation.
\end{remark}
\begin{remark}\label{remark 2}
 For the case of $\lambda=0$, the convergence rates shown in Theorem \ref{Thm 1}-\ref{Thm 1-1} and Corollary \ref{corollary 1} are the same as
 all existing convergence rates obtained in the previous works  \cite{Jiang-Zhu2009,Marcati-Mei2000,Marcati-Mei-Rubino2005,Nishihara-Yang1999}.
\end{remark}

\bigbreak

\subsection{Preliminaries}\label{S2.2}

In this subsection, we will establish some fundamental dissipative properties of the solution $(\bar v, \bar u)(x,t)$ to the system \eqref{a1.5}. The equations \eqref{a1.5} can also be written as
\begin{equation}\notag
\left\{\begin{array}{l}
\partial_t  \bar{v}
  +\displaystyle
  \frac{(1+t)^\lambda}{\alpha}
   \partial_{xx} p(\bar{v})=0,\qquad (x,t)\in \mathbb{R}^+\times\mathbb{R}^+,\\[2mm]
     \bar v\mid_{t=0}=\bar v_0(x)\rightarrow v_+,\quad
    \mbox{as}
     \quad
     x\rightarrow  +\infty,\\[2mm]
     \bar{v}_x( 0,t)=0, \qquad \bar{v}(  \infty,t)=v_+.
 \end{array}
        \right.
\end{equation}

 Let $(\bar v_0^\star, \bar u_0^\star)(x)$ denote the even and odd extensions of  $(\bar v_0, \bar u_0)(x)$ in the whole space $\mathbb{R}$, respectively,
 i.e.,
 \begin{equation}\notag
\bar v_0^\star(x)=
\left\{
\begin{array}{l}
\bar  v_0(x),\qquad \ \ \mbox{if} \qquad x\geq0,\\
\bar v_0(-x),\qquad\mbox{if} \qquad x<0,
\end{array}
\right.
\qquad
\bar u_0^\star(x)=
\left\{
\begin{array}{l}
\bar  u_0(x),\qquad\ \  \ \  \mbox{if} \qquad x\geq0,\\
-\bar u_0(-x),\qquad  \mbox{if} \qquad x<0.
\end{array}
\right.
\end{equation}
 Then, we study the properties of $(\bar v,\bar u)(x,t)$ by investigating the following Cauchy problem:
 \begin{equation}\label{a1.7}
\left\{\begin{array}{l}
\partial_t  \bar{v}
  -
   \partial_x \bar{ u}=0,\\[2mm]
      \partial_x  p(\bar{v})
     =\displaystyle
     -\frac{\alpha}{(1+t)^\lambda} \bar{u},\qquad (x,t)\in \mathbb{R}\times\mathbb{R}^+,\\[3mm]
     (\bar v,\bar u)\mid_{t=0}=(\bar v_0^\star,\bar u_0^\star)(x)\rightarrow (v_+,0),\quad
    \mbox{as}
     \quad
     x\rightarrow  \pm\infty.
 \end{array}
        \right.
\end{equation}
Different from the Cauchy problem in \cite{Cui-Yin-Zhang-Zhu2016}, the
initial value problem \eqref{a1.7} does not possess self-similar solution.
Therefore,  we can't directly get the decay rates of $(\bar v,\bar
u)(x,t)$. Therefore, as in \cite{Jiang-Zhu20092}, we study the corresponding linearized
problem of \eqref{a1.7} around $v_+$
 \begin{equation}\label{a1.a7}
\left\{\begin{array}{l}
\partial_t  \tilde{v}
  -
   \partial_x \tilde{ u}=0,\\[2mm]
      p'(v_+) \partial_x \tilde{v}
     =\displaystyle
     -\frac{\alpha}{(1+t)^\lambda} \tilde{u},\qquad (x,t)\in \mathbb{R}\times\mathbb{R}^+,\\[2mm]
     \tilde v(x,0)=v_++\frac{\delta_0(\lambda+1)^\frac{1}{2}}{ (4\kappa\pi)^{\frac{1}{2}} }e^{-\frac{(\lambda+1)x^2}{4\kappa }},\\[2mm]
      \tilde u(x,0)=-\frac{ x}{2 }\frac{\delta_0(\lambda+1)^\frac{3}{2}}{ (4\kappa\pi)^{\frac{1}{2}} }e^{-\frac{(\lambda+1)x^2}{4\kappa }},
 \end{array}
        \right.
\end{equation}
 where $\kappa=-\frac{p'(v_+)}{\alpha}>0$ and $\delta_0=2\int_0^\infty(\bar v_0(y)-v_+)dy$.

 The solution $(\tilde v,\tilde u)(x,t)$ of the Cauchy problem \eqref{a1.a7} can be  written explicitly as
 \begin{equation}\notag
 \left\{
 \begin{array}{l}
 \tilde v(x,t)=v_++\frac{\delta_0(\lambda+1)^\frac{1}{2}}{ (4\kappa\pi)^{\frac{1}{2}} (1+t)^{\frac{\lambda+1}{2}}}e^{-\frac{(\lambda+1)x^2}{4\kappa (1+t)^{\lambda+1}}},\\[4mm]
 \tilde u(x,t)=\displaystyle \kappa(1+t)^\lambda \tilde v_x(x,t),\qquad (x,t)\in \mathbb{R}\times\mathbb{R}^+.
\end{array}
\right.
\end{equation}

By direct calculations, we have the following estimates.
\begin{lemma}\label{Lemma a0.1}
For each $p\in[1,\infty]$, we know
\begin{equation}\notag
\begin{split}
&
       \|\partial_t^l\partial_x^k(\tilde v(t)-\bar v_+)\|_{L^p(\mathbb{R})}
            \leq
                   C|\delta_0| (1+t)^{-\frac{\lambda+1}{2}(1-\frac{1}{p})-\frac{(\lambda+1)k}{2}-l},\qquad k,l=0,1,2,\cdot\cdot\cdot.
\end{split}
\end{equation}
Furthermore, for  each $p\in[1,\infty]$, let
$h(x,t)=-(p'(v_+)-p'(\tilde v))\tilde v_x$, we have
\begin{equation}\notag
\int_{-\infty}^\infty |\partial_t^l\partial_x^k h(x,t)|^2 dx\leq C\delta_0^4 (1+t)^{-\frac{5(\lambda+1)}{2}- (\lambda+1)k -2l}, \qquad\qquad k,l=0,1,2,\cdot\cdot\cdot.
\end{equation}
\end{lemma}

Combining \eqref{a1.7} and \eqref{a1.a7} leads to
\begin{equation}\label{a1.12}
\left\{\begin{array}{l}
\partial_t  (\bar{v}-\tilde{v})
  -
   \partial_x (\bar{ u}-\tilde{u})=0,\qquad (x,t)\in \mathbb{R}\times\mathbb{R}^+,\\[2mm]
      \partial_x  (p(\bar v)-p(\tilde{v}))
           +\displaystyle
              \frac{\alpha}{(1+t)^\lambda} (\bar{u}-\tilde{u})=(p'(v_+)-p'(\tilde v))\tilde v_x.
 \end{array}
        \right.
\end{equation}
Integrating of $(\ref{a1.12})_1$ with respect to  $x$ and $t$ over $ \mathbb{R}\times(0,t)$ and using \eqref{a1.61}, we  can get
\begin{equation}\notag
\begin{split}
\int_{-\infty}^\infty( \bar v(x ,t)-\tilde v(x,t))dx
&
  =\int_{-\infty}^\infty( \bar v_0^\star(x)-\tilde v (x,0))dx\\
 &
  =2\int_{0}^\infty( \bar v_0(x)-\tilde v(x,0))dx\\
  &
  =2\int_{0}^\infty (\bar v_0-v_+)dx- {\delta_0}\\
  &
   =0.
\end{split}
\end{equation}
Hence we define the new variables
\begin{equation}\label{a1.14}
V(x,t)=\int_{-\infty}^{x} ( \bar v(y,t)-\tilde{v}(y,t))dy,
\end{equation}
and
\begin{equation}\label{a1.15}
Z(x,t)=  \bar u(x,t)-\tilde{u}(x,t).
\end{equation}
By \eqref{a1.12}, we have the reformulated problem
\begin{equation}\label{a1.19}
\left\{\begin{array}{l}
  V_t
  -
   Z=0,
   \qquad \qquad \qquad \qquad  (x,t)\in \mathbb{R}\times\mathbb{R}^+,\\[2mm]
        (p(V_x +\tilde v)-p(\tilde{v}))_x
           +\displaystyle
              \frac{\alpha}{(1+t)^\lambda} Z= (p'(v_+)-p'(\tilde v))\tilde v_x.
 \end{array}
        \right.
\end{equation}
The corresponding  initial data are given by
\begin{equation}\label{a1.19idata}
  (V,Z)\mid_{t=0}=(V_0,Z_0)(x),
\end{equation}
where
\begin{equation}\label{a1.19idata2}
\left\{\begin{array}{l}
  \ds V_0(x)
    =\int_{-\infty}^x ( \bar v_0^\star(y)-\tilde{v}(y,0))dy,\\[3mm]
   \ds Z_0(x)
       = \bar u_0^\star(x)-\tilde{u}(x,0).
\end{array}
\right.
\end{equation}
Rewrite  \eqref{a1.19} and \eqref{a1.19idata} as
\begin{equation}\label{a1.20}
\begin{array}{l}
        (p'(\tilde{v})V_x)_x
           +\displaystyle
              \frac{\alpha}{(1+t)^\lambda} V_t=-F_1,
              \qquad (x,t)\in \mathbb{R}\times\mathbb{R}^+,
              \end{array}
\end{equation}
with initial data
\begin{equation}\label{a1.20idata}
  (V,V_t)\mid_{t=0}=(V_0,Z_0)(x),
\end{equation}
where
\begin{equation}\label{a1.21}
\begin{split}
  F_1
   =
    h(x,t)
         +
            (p(V_x+\tilde v)-p(\tilde{v})-p'(\tilde v)V_x)_x.
 \end{split}
\end{equation}
Noticing that, by \eqref{a1.61} and the definition of $\delta_0$, we have
\begin{equation}\label{1.20idata}
 | \delta_0|\leq C(\|  v_0-v_+\|_{L^1(\mathbb{R}^+)}+|u_+|) .
\end{equation}
In order to get the dissipative properties of $(\bar v,\bar u)(x,t)$, now we investigate the properties of $(V,Z)(x,t)$. In fact, we could obtain the following theorem.
\begin{theorem}\label{Thm a1}
For any $\alpha>0$, $(V_0,Z_0)(x)\in H^3(\mathbb{R})\times H^2(\mathbb{R})$, assume that
  $\|  v_0-v_+\|_{L^1(\mathbb{R}^+)} +|u_+|+\|V_0\|_{H^3(\mathbb{R})}+\|Z_0\|_{H^2(\mathbb{R})}$ is sufficiently small.  Then, for any $0\leq\lambda<1$,
there exists a unique time-global solution $(V,Z)(x,t)$ of the Cauchy
problem \eqref{a1.20}-\eqref{a1.21} satisfying
\begin{equation}\notag
  V\in C^{k}((0,\infty),H^{3-k}(\mathbb{R})),\qquad k=0,1,2,3,
\end{equation}
\begin{equation}\notag
  V_t\in C^{k}((0,\infty),H^{2-k}(\mathbb{R})),\qquad k=0,1,2.
\end{equation}
Furthermore, we have
\begin{equation}\notag
\begin{split}
     &
      \sum_{k=0}^3(1+t)^{ (\lambda +1)k}
      \|\partial_x^kV(\cdot,t)\|_{L^2(\mathbb{R})}^2
      +
         \sum_{k=0}^2(1+t)^{(\lambda+1)k +2} \|\partial_x^kV_t(\cdot,t)\|_{L^2(\mathbb{R})}^2\\
                 \leq
                       &
                          C(\|V_0\|_{H^3(\mathbb{R})}^2
                          +\|Z_0\|_{H^2(\mathbb{R})}^2+|\delta_0|^2),
\end{split}
\end{equation}
and
\begin{equation}\notag
\begin{split}
     &
       (1+t)^{4}
      \| Z_t(\cdot,t)\|_{L^2(\mathbb{R})}^2
      +
          (1+t)^{ \lambda+5} \| Z_{xt}(\cdot,t)\|_{L^2(\mathbb{R})}^2\\
                 \leq
                       &
                          C(\|V_0\|_{H^3(\mathbb{R})}^2+\|Z_0\|_{H^2(\mathbb{R})}^2
                          +|\delta_0|^2).
\end{split}
\end{equation}
\end{theorem}

{\bf Notations}: For the sake of simplicity, throughout this subsection,  we denote  $\|\cdot\|_{L^p}:=\|\cdot\|_{L^p(\mathbb{R})}$, $1\leq p\leq \infty$, $\|\cdot\|:=\|\cdot\|_{L^2(\mathbb{R})}$, $\|\cdot\|_{H^l}:=\|\cdot\|_{H^l(\mathbb{R})}$,  $l\geq 0$, and we also use $\ds \int f dx:=\int_{\mathbb{R}}fdx$.

Now, we  begin to estimate the solution $(V,Z)(x,t)$, $0<t<T<\infty$, to the Cauchy problem \eqref{a1.20}-\eqref{a1.21}  under the {\it a priori} assumption
\begin{equation}\label{axian}
N_1(T):=\sup_{0<t<T} \bigg\{\| V_x(\cdot,t)\|_{L^\infty}+ (1+t)\|
V_{xt}(\cdot,t)\|_{L^\infty}+ (1+t)^{\frac{\lambda+1}{2}}\|
V_{xx}(\cdot,t)\|_{L^\infty}\bigg\}\leq \epsilon
\end{equation}
for some $0<\epsilon\ll 1$.\\

Then we will establish some necessary {\it a priori } bounds for $(V,Z)$. The first result is the lower order energy estimates.
\begin{lemma}\label{Lemma a1}
Under the assumptions of Theorem \ref{Thm a1}, if $\epsilon, |\delta_0|$ are small, it holds that
\begin{equation}\notag
\begin{split}
       \|V\|^2+(1+t)^{\lambda+1}\|V_x\|^2
       +
       \int_0^t\big\{(1+s)\|V_t\|^2+(1+s)^{\lambda}\|V_x\|^2\big\}ds
\leq
   C( \|V_0\|_{H^1}^2+|\delta_0|^2).
\end{split}
\end{equation}
\end{lemma}
\begin{proof}
First, multiplying \eqref{a1.20} by $(1+t)^\lambda V$, integrating the
resulting  equality with respect to $x$ over $\mathbb{R} $ give
\begin{equation}\label{a3.3}
\begin{split}
\frac{\rm d}{{\rm d}t} \int
   \frac{\alpha}{2} V^2dx
  - \int(1+t)^{\lambda} p'(\tilde v)V_x^2dx
=
     -\int (1+t)^{\lambda} F_1Vdx.
\end{split}
\end{equation}
Then,  to estimate the last term in the  right hand of \eqref{a3.3}, one obtains that
\begin{equation}\label{a3.4}
\begin{split}
&-\int (1+t)^{\lambda} F_1Vdx\\
=&-\int (1+t)^\lambda h(x,t)Vdx+\int (1+t)^\lambda\big(p(V_x+\tilde v)-p(\tilde{v})-p'(\tilde v)V_x\big)V_xdx\\
\leq
   &
   C\int(1+t)^{-\kappa}V^2 dx+\int(1+t)^{2\lambda+\kappa}h^2 dx+\frac{1}{2}\int(1+t)^\lambda p''(\theta_1V_x+\tilde v)V_x^3dx\\
\leq
&
 C\int(1+t)^{-\kappa}V^2 dx
 +
    C|\delta_0| ^2(1+t)^{\kappa-\frac{\lambda}{2}-\frac{5}{2}}+C\epsilon\int(1+t)^\lambda V_x^2dx,
\end{split}
\end{equation}
where $0<\theta_1<1$.

Substituting \eqref{a3.4} into  \eqref{a3.3},   using the smallness of $\epsilon$ and for some positive constant $C_0$ satisfying $-p'(\tilde v)\geq C_0>0$, we have
\begin{equation}\notag
\begin{split}
\frac{\rm d}{{\rm d}t} \int
   \frac{\alpha}{2}V^2dx
  +
  \frac{C_0}{2}\int(1+t)^{\lambda}V_x^2dx
\leq
      C\int(1+t)^{-\kappa}V^2 dx
 +
    C|\delta_0| ^2(1+t)^{\kappa-\frac{\lambda}{2}-\frac{5}{2}}.
\end{split}
\end{equation}
Then taking $1<\kappa<\frac{\lambda}{2}+\frac{3}{2}$  and using  Gronwall’s inequality lead to
\begin{equation}\label{a3.8}
\|V\|^2
       +
       \int_0^t(1+s)^{\lambda}\|V_x\|^2ds\\
\leq
   C( \|V_0\|^2+|\delta_0|^2).
\end{equation}

Next, multiplying \eqref{a1.20} by $(1+t)^{2\lambda} V_t$, and
integrating the resulting equality over $\mathbb{R}$, one yields
\begin{equation}\label{a3.9}
\begin{split}
&
-\frac{1}{2}\frac{\rm d}{{\rm d}t} \int
(1+t)^{2\lambda}p'(\tilde v)V_x^2dx
  + \int\alpha(1+t)^{\lambda}V_t^2dx
  \\
=
 &
  -
    \frac{1}{2}\int (1+t)^{2\lambda} p''(\tilde v)\tilde v_tV_x^2dx
    -
    \lambda\int(1+t)^{2\lambda-1}p'(\tilde v)V_x^2 dx
     -
       \int (1+t)^{2\lambda} F_1V_tdx\\
        \leq
           &
            C\int(1+t)^{2\lambda-1}V_x^2 dx
    -
       \int (1+t)^{2\lambda} F_1V_tdx.
\end{split}
\end{equation}
Now we estimate the last term in the  right hand of \eqref{a3.9} as follows:
\begin{equation}\label{a3.10}
\begin{split}
-\int (1+t)^{2\lambda} F_1V_tdx =
&
 -\int(1+t)^{2\lambda}
h(x,t)V_tdx \\
& +\int
            (1+t)^{2\lambda}\big(p(V_x+ \tilde v)-p(\tilde{v})-p'(\tilde v)V_x\big)V_{xt}dx.
\end{split}
\end{equation}
Firstly, applying Lemma \ref{Lemma a0.1}, one gets
\begin{equation}\label{a3.11}
\begin{split}
&
 \int  (1+t)^{2\lambda} h(x,t)V_tdx
   \\
    \leq
     &
       \frac{\alpha}{2}\int(1+t)^{\lambda}  V_t^2dx
       +
       C\int (1+t)^{3\lambda} h^2dx \\
          \leq
            &
            \frac{\alpha}{2}\int(1+t)^{\lambda}  V_t^2dx
            +
            C|\delta_0|^2(1+t)^{\frac{\lambda}{2}-\frac{5}{2}}.
\end{split}
\end{equation}
Next by using Lemmas \ref{Lemma a0.1} and the {\it a
priori} assumption \eqref{axian}, then we get
\begin{equation}\label{a3.13}
\begin{split}
       &
        \int (1+t)^{2\lambda}
        \Big(
        p(V_x+\tilde v)-p(\tilde{v})-p'(\tilde v)V_x
        \Big)
        V_{xt}dx
        \\
         =
         &
           (1+t)^{2\lambda}\frac{\rm d}{{\rm d}t}\int
           \bigg(
           \int_{\tilde v}^{V_x+\tilde v} p(s)ds
           -
           p(\tilde v)V_x
           -
           \frac{p'(\tilde v)}{2}V_x^2 \bigg)dx\\
           &
           -
            \int(1+t)^{ 2\lambda} \tilde v_t\Big(p(V_x+\tilde v)-p(\tilde{v})-p'(\tilde v)V_x-\frac{p''(\tilde v)}{2}V_x^2\Big)dx\\
              \leq
              &
                \frac{\rm d}{{\rm d}t}
                \int
                (1+t)^{2\lambda}\bigg(
                \int_{\tilde v}^{V_x+\tilde v} p(s)ds
                -
                p(\tilde v)V_x
                -
                \frac{p'(\tilde v)}{2}V_x^2
                \bigg)dx \\
        &
        -
          2\lambda\int(1+t)^{2\lambda-1}  \bigg(  \int_{\tilde v}^{V_x+\tilde v} p(s)ds-p(\tilde v)V_x-\frac{p'(\tilde v)}{2}V_x^2 \bigg)dx\\
          &
           +C|\delta_0|\int (1+t)^{\frac{3\lambda}{2}-\frac{3}{2}}|V_x|^3dx \\
            \leq
              &
               \frac{\rm d}{{\rm d}t}\int (1+t)^{2\lambda}\bigg(  \int_{\tilde v}^{V_x+\tilde v} p(s)ds-p(\tilde v)V_x-\frac{p'(\tilde v)}{2}V_x^2 \bigg)dx\\
               &+C(|\delta_0|+\epsilon)\int(1+t)^{2\lambda-1}  V_x^2dx.
\end{split}
\end{equation}
Since
\begin{equation}\notag
\int_{\tilde v}^{V_x+\tilde v} p(s)ds=p(\tilde v)V_x+\frac{1}{2}p'(\theta_2V_x+\tilde v)V_x^2,
\end{equation}
where $0<\theta_2<1$.
Then, putting  \eqref{a3.10}-\eqref{a3.13} into \eqref{a3.9}, we have
\begin{equation}\label{a3.14}
\begin{split}
&
-\frac{1}{2}\frac{\rm d}{{\rm d}t} \int
(1+t)^{2\lambda}p'(\theta_2V_x+\tilde v)V_x^2
 dx
  + \frac{\alpha}{2}\int (1+t)^{\lambda}V_t^2dx
  \\
\leq
   &
    C \int  (1+t)^{2\lambda-1} V_x^2dx
        +
         C|\delta_0|^2(1+t)^{ \frac{\lambda}{2}-\frac{5}{2}}.
\end{split}
\end{equation}
It follows from  \eqref{a3.8} and $0\leq\lambda<1$ that
\begin{equation}\notag
     (1+t)^{2\lambda}\|V_x\|^2
       +
       \int_0^t(1+s)^{\lambda}\|V_t\|^2ds
\leq
   C( \|V_0\|_{H^1}^2+|\delta_0|^2).
\end{equation}
Finally, multiplying  \eqref{a3.14} by $(1+t)^{1-\lambda}$, and using \eqref{a3.8} again,  we know
\begin{equation}\notag
     (1+t)^{\lambda+1}\|V_x\|^2
       +
       \int_0^t(1+s)\|V_t\|^2ds
\leq
   C( \|V_0\|_{H^1}^2+|\delta_0|^2).
\end{equation}
Hence we complete the proof of Lemma \ref{Lemma a1}.
\end{proof}
\begin{lemma}\label{Lemma a2}
Under the assumptions of Theorem \ref{Thm a1}, if $\epsilon, |\delta_0|$ are small, it holds that
\begin{equation}\notag
\begin{split}
   (1+t)^{ 2\lambda+2}\|V_{xx}\|^2+ \int_0^t \Big((1+s)^{\lambda+2}\|V_{xt}\|^2+(1+s) ^{2\lambda+1} \|V_{xx}\|^2\Big) ds
       \leq
   C ( \|V_0\|_{H^2}^2+|\delta_0|^2).
\end{split}
\end{equation}
\end{lemma}
\begin{proof}
Differentiating \eqref{a1.20} with respect to $x$, one yields
\begin{equation}\label{a3.32}
(p'(\tilde{v})V_x)_{xx}
           +\displaystyle
              \frac{\alpha}{(1+t)^\lambda} V_{xt}=-F_{1x}.
\end{equation}
Multiplying \eqref{a3.32} by $(1+t)^{2\lambda} V_{xt}$ and
integrating the resulting equality with respect to $x$ over
$\mathbb{R}$, we have by using integrations by parts that
\begin{equation}\label{a3.33}
\begin{split}
&
-\frac{1}{2}\frac{\rm d}{{\rm d}t} \int
(1+t)^{2\lambda}p'(\tilde v)V_{xx}^2dx
  +
  \int\alpha(1+t)^{\lambda}V_{xt}^2dx
  \\
=
 &
   -
     \lambda \int(1+t)^{2\lambda-1}p'(\tilde v)V_{xx}^2 dx
     -
     \frac{1}{2}\int (1+t)^{2\lambda} p''(\tilde v)\tilde v_{t}V_{xx}^2dx\\
     &
     -
     \int (1+t)^{2\lambda}p'''(\tilde v)|\tilde v_x|^2 V_{x}V_{xt} dx
      -
      \int (1+t)^{2\lambda}p''(\tilde v)\tilde v_{xx}V_{x}V_{xt} dx\\
     &
     -
     \int (1+t)^{2\lambda}p''(\tilde v)\tilde v_x V_{xx}V_{xt} dx
     -\int (1+t)^{2\lambda} F_{1x}V_{xt}dx
     :=
   \sum_{k=1}^6 I_k.
\end{split}
\end{equation}
We utilize Cauchy-Schwarz's inequality and Lemma \ref{Lemma a0.1} to address the following estimates:
\begin{equation}\label{ai3.2}
\begin{split}
I_1\leq C\int (1+t)^{2\lambda-1}V_{xx}^2dx,
\end{split}
\end{equation}
\begin{equation}\label{ai3.3}
\begin{split}
I_2
\leq
C|\delta_0| \int (1+t)^{\frac{3\lambda}{2}-\frac{3}{2}}V_{xx}^2dx,
\end{split}
\end{equation}
\begin{equation}\label{ai3.4}
\begin{split}
I_3\leq
&
\frac{\alpha}{16}\int(1+t)^{\lambda}  V_{xt}^2dx
         +
          C\int (1+t)^{3\lambda}|\tilde v_x|^4V_{x}^2dx\\
          \leq
           &
          \frac{\alpha}{16}\int(1+t)^{\lambda}  V_{xt}^2dx
         +
          C|\delta_0|^2 \int (1+t)^{-\lambda-4}V_{x}^2dx,
\end{split}
\end{equation}
\begin{equation}\label{ai3.5}
\begin{split}
I_4\leq
&
\frac{\alpha}{16}\int(1+t)^{\lambda}  V_{xt}^2dx
         +
          C\int (1+t)^{3\lambda}|\tilde v_{xx}|^2V_{x}^2dx\\
          \leq
           &
          \frac{\alpha}{16}\int(1+t)^{\lambda}V_{xt}^2dx
         +
          C|\delta_0|^2 \int (1+t)^{-3}V_{x}^2dx,
\end{split}
\end{equation}
and
\begin{equation}\label{ai3.6}
\begin{split}
I_5\leq &\frac{\alpha}{16}\int(1+t)^{\lambda}  V_{xt}^2dx
         +
          C\int (1+t)^{3\lambda}|\tilde v_x|^2V_{xx}^2dx\\
          \leq
           &
          \frac{\alpha}{16}\int(1+t)^{\lambda}  V_{xt}^2dx
         +
          C|\delta_0|^2\int (1+t)^{\lambda-2}V_{xx}^2dx.
\end{split}
\end{equation}
Now we turn to estimate $I_6$ as follows:
\begin{equation}\label{a3.34}
\begin{split}
I_6
=&\int \bigg[-(1+t)^{2\lambda} h_x-
            (1+t)^{2\lambda}\big(p(V_x+\tilde v)-p(\tilde{v})-p'(\tilde v)V_x\big)_{xx}\bigg]V_{xt}dx.
\end{split}
\end{equation}
By employing Lemmas \ref{Lemma a0.1} and the {\it a priori} assumption \eqref{axian}, we can get
\begin{equation}\label{a3.35}
\begin{split}
&
 -\int (1+t)^{2\lambda} h_{x}V_{xt}dx
   \\
    \leq
     &
       \frac{\alpha}{16}\int(1+t)^{\lambda}  V_{xt}^2dx
         +
          C\int (1+t)^{3\lambda }h_x^2dx \\
          \leq
            &
            \frac{\alpha}{16}\int(1+t)^{\lambda}  V_{xt}^2dx
            +
            C|\delta_0|^2(1+t)^{-\frac{\lambda}{2}-\frac{7}{2}},
\end{split}
\end{equation}
and
\begin{equation}\label{a3.37}
\begin{split}
   &
    -\int (1+t)^{2\lambda}\bigg(p(V_x+\tilde v)-p(\tilde{v})-p'(\tilde v)V_x\bigg)_{xx}V_{xt}dx\\
      =
       &
        \int (1+t)^{2\lambda}\bigg(p(V_x+\tilde v)-p(\tilde{v})-p'(\tilde v)V_x\bigg)_x V_{xxt}dx\\
=
         &
           (1+t)^{2\lambda}\frac{1}{2}\frac{\rm d}{{\rm d}t}\int\bigg(p'(V_x+\tilde v)-p'(\tilde v) \bigg)V_{xx}^2 dx\\
           &
           -
           \frac{1}{2}\int (1+t)^{2\lambda}\bigg(p''(V_x+\tilde v)(V_{xt}+\tilde v_t )-p''(\tilde v)\tilde v_t\bigg)V_{xx}^2dx\\
           &
           +
            \int(1+t)^{2\lambda} \tilde v_xV_{xxt}\big(p'(V_x+\tilde v)-p'(\tilde{v})-p''(\tilde v)V_x\big)dx\\
\leq
&
               \frac{1}{2}\frac{\rm d}{{\rm d}t}\int (1+t)^{2\lambda}\bigg(p'(V_x+\tilde v)-p'(\tilde v) \bigg)V_{xx}^2 dx \\
        &
        -
         \lambda\int(1+t)^{2\lambda-1}  \bigg(p'(V_x+\tilde v)-p'(\tilde v) \bigg)V_{xx}^2dx
          +C(\epsilon+|\delta_0|)\int(1+t)^{2\lambda-1}  V_{xx}^2dx\\
          &
         -
            \int(1+t)^{2\lambda} \tilde v_{xx}V_{xt}\big(p'(V_x+\tilde v)-p'(\tilde{v})-p''(\tilde v)V_x\big)dx\\
\leq &
               \frac{1}{2}\frac{\rm d}{{\rm d}t}\int (1+t)^{2\lambda}\bigg(p'(V_x+\tilde v)-p'(\tilde v) \bigg)V_{xx}^2 dx
               +
               \frac{\alpha}{16}\int(1+t)^{\lambda}  V_{xt}^2dx\\
               &+C(\epsilon+|\delta_0|)\int(1+t)^{2\lambda-1}  V_{xx}^2dx+C|\delta_0|^2\int(1+t)^{  -3}  V_{x}^2dx.
\end{split}
\end{equation}
Substituting \eqref{ai3.2}-\eqref{a3.37} into \eqref{a3.33},  using
Lemma \ref{Lemma a1}, $0\leq\lambda<1$ and taking $\epsilon, |\delta_0|$ sufficiently small,
we derive
\begin{equation}\label{a3.38}
\begin{split}
&
-\frac{1}{2}\frac{\rm d}{{\rm d}t} \int
(1+t)^{2\lambda}p'(V_x+\tilde v)V_{xx}^2dx
  + \frac{\alpha}{2}\int(1+t)^{\lambda} V_{xt}^2dx
  \\
\leq
   &
    C \int  (1+t)^{2\lambda-1}  V_{xx}^2 dx
    +
    C|\delta_0|^2\int(1+t)^{-3}  V_{x}^2dx
        +
        C|\delta_0|^2(1+t)^{-\frac{\lambda}{2}-\frac{7}{2}}\\
\leq
   &
    C \int  (1+t)^{2\lambda-1} V_{xx}^2 dx
        +
         C(1+t)^{-\frac{\lambda}{2}-\frac{7}{2}}( \|V_0\|_{H^1}^2+|\delta_0|^2).
\end{split}
\end{equation}
Next, multiplying \eqref{a1.20} by $-(1+t)^\lambda V_{xx}$, integrating the
resulting equality with respect to $x$ over $\mathbb{R}$, and using
$0\leq \lambda<1$, we obtain
\begin{equation}\label{a3.40}
\begin{split}
\frac{\rm d}{{\rm d}t} \int
   \frac{\alpha}{2}V_x^2dx
  - \int(1+t)^{\lambda}  p'(\tilde v)V_{xx}^2dx
=
      \int (1+t)^\lambda p''(\tilde v )\tilde v_x V_{x} V_{xx}dx
      +
       \int (1+t)^{\lambda}  F_1V_{xx}dx.
\end{split}
\end{equation}
By Lemmas \ref{Lemma a0.1}-\ref{Lemma a1}, it is easy to see that
\begin{equation}\label{a3.41}
\begin{split}
\int (1+t)^{\lambda}  p''(\tilde v )\tilde v_x V_x V_{xx} dx
\leq
&
\frac{C_0}{10}\int(1+t)^{\lambda}   V_{xx}^2dx
        +
        C\int(1+t)^{\lambda}   |\tilde v_x|^2V_{x}^2dx\\
 \leq
 &
 \frac{C_0}{10}\int(1+t)^{\lambda}   V_{xx}^2dx
        +
        C|\delta_0|^2\int(1+t)^{-\lambda-2} V_{x}^2dx\\
 \leq
 &
 \frac{C_0}{10}\int(1+t)^{\lambda}   V_{xx}^2dx
        +
        C (1+t)^{-2\lambda -3} (\|V_0\|_{H^1}^2+|\delta_0|^2) .
\end{split}
\end{equation}
From \eqref{a1.21}, one has
\begin{equation}\label{a3.42}
\begin{split}
\int (1+t)^{\lambda}  F_1V_{xx}dx
=&\int \bigg[ (1+t)^{\lambda} h
     +
            (1+t)^\lambda\big(p(V_x+\tilde v )-p(\tilde{v})-p'(\tilde v)V_x\big)_x\bigg]V_{xx}dx.
\end{split}
\end{equation}
By using Lemmas \ref{Lemma a0.1}-\ref{Lemma a1} and the {\it a priori} assumption \eqref{axian}, we have
\begin{equation}\label{a3.43}
\begin{split}
        \int  (1+t)^{\lambda }h(x,t)V_{xx}dx
\leq&
         \frac{C_0}{10}\int(1+t)^{\lambda}   V_{xx}^2dx+C\int (1+t)^{\lambda }h^2dx \\
\leq
           &
              \frac{C_0}{10}\int(1+t)^{\lambda}   V_{xx}^2dx+C|\delta_0|^2(1+t)^{-\frac{3\lambda}{2}-\frac{5}{2}},
\end{split}
\end{equation}
and
\begin{equation}\label{a3.45}
\begin{split}
&
\int (1+t)^{\lambda}\big(p(V_x+\tilde v )-p(\tilde{v})-p'(\tilde v)V_x\big)_xV_{xx}dx\\
   =
    &
      \int (1+t)^{\lambda} \big(p'(V_x+\tilde v )-p'(\tilde{v})\big)V_{xx}^2dx\\
    &
   +
   \int (1+t)^{\lambda}\big(p'(V_x+\tilde v )-p'(\tilde{v})-p''(\tilde v)V_x\big)\tilde v_x V_{xx}dx\\
    \leq
       &
            C\int (1+t)^{\lambda} |V_x| V_{xx}^2dx
       +\frac{C_0}{10}\int(1+t)^{\lambda}   V_{xx}^2dx
       +C\int (1+t)^{\lambda} |\tilde v_x|^2|V_x|^4dx\\
           \leq
             &
               C\epsilon\int(1+t)^{\lambda}   V_{xx}^2dx
             +\frac{C_0}{10}\int(1+t)^{\lambda}   V_{xx}^2dx
             +
                 C\epsilon^2 |\delta_0|  \int(1+t)^{-\lambda-2}  V_x^2dx\\
           \leq
             &
               C\epsilon\int(1+t)^{\lambda}   V_{xx}^2dx
             +
             \frac{C_0}{10}\int(1+t)^{\lambda}   V_{xx}^2dx
             +
                 C(1+t)^{-{2}\lambda-3} (\|V_0\|_{H^1}^2+|\delta_0|^2).
\end{split}
\end{equation}
Substituting \eqref{a3.41}-\eqref{a3.45} into \eqref{a3.40}, and $0\leq\lambda<1$ and the smallness of $\epsilon, |\delta_0|$, we have
\begin{equation}\label{a3.46}
\begin{split}
\frac{\rm d}{{\rm d}t}
 \int
     \frac{\alpha}{2}  V_x^2dx
  +
  \frac{C_0}{2}\int(1+t)^{\lambda} V_{xx}^2dx
\leq
         C(1+t)^{-\frac{3\lambda}{2}-\frac{5}{2}} (\|V_0\|_{H^1}^2+|\delta_0|^2),
\end{split}
\end{equation}
which together with \eqref{a3.38} yields
\begin{equation}\notag
(1+t)^{2\lambda} \|V_{xx}\|^2+\int_0^t (1+s)^{\lambda}(\|V_{xt}\|^2+\|V_{xx}\|^2)ds
       \leq
   C( \|V_0\|_{H^2}^2+|\delta_0|^2).
\end{equation}
Finally, performing  \eqref{a3.46}$\times(1+t)^{1+\lambda}$, \eqref{a3.38}$\times(1+t)^{2}$, respectively, and using the above inequality,  we  complete the proof of Lemma \ref{Lemma a2}.
\end{proof}

Similar to the proof of Lemma \ref{Lemma a2}, performing $(\ref{a1.20})_{xx}\times (1+t)^{2\lambda}V_{xxt}$, $(\ref{a1.20})_{x}\times (1+t)^{\lambda} (-V_{xxx})$, integrating the
resulting  equations with respect to $x$ over $\mathbb{R} $, we can get the following lemma:
\begin{lemma}\label{Lemma a3}
Under the assumptions of Theorem \ref{Thm a1}, if $\epsilon, |\delta_0|$ are small, it holds that
\begin{equation}\notag%\label{a3.31}
\begin{split}
   (1+t)^{ 3\lambda+3}\|V_{xxx}\|^2+ \int_0^t \Big((1+s)^{2\lambda+3}\|V_{xxt}\|^2+(1+s) ^{3\lambda+2} \|V_{xxx}\|^2)\Big) ds
       \leq
   C ( \|V_0\|_{H^3}^2+|\delta_0|^2).
\end{split}
\end{equation}
\end{lemma}

\begin{lemma}\label{Lemma a5}
Under the assumptions of Theorem \ref{Thm a1}, if $\epsilon, |\delta_0|$ are small, it holds that
\begin{equation}\label{a3.51}
\begin{split}
 & (1+t)^{2}\|Z\|^2
        +
        (1+t)^{\lambda+3}\|Z_{x}\|^2+ \int_0^t \Big((1+s)^{3}\|Z_{t}\|^2+(1+s)^{\lambda+2}\|Z_{x}\|^2\Big)ds\\
       \leq   &
   C
   ( \|V_0\|_{H^2}^2+\|Z_0 \|_{H^1}^2+|\delta_0|^2).
\end{split}
\end{equation}

\end{lemma}
\begin{proof}
Differentiating \eqref{a1.20} with respect to $t$ leads to
\begin{equation}\label{a3.52}
(p'(\tilde{v})V_x)_{xt}
           +\displaystyle
              \frac{\alpha}{(1+t)^\lambda} Z_t
              =
                \frac{\alpha\lambda }{(1+t)^{\lambda+1}}Z-F_{1t}.
\end{equation}
 Multiplying \eqref{a3.52} by $(1+t)^{\lambda}  Z$, from Lemma
\ref{Lemma a1} and after complicated calculations, we know
\begin{equation}\label{a3.54}
\begin{split}
&
\frac{\rm d}{{\rm d}t} \int
   \frac{\alpha}{2}  Z^2dx
  +
  \frac{C_0}{2}\int(1+t)^{\lambda} Z_x^2dx
  \\
\leq
 &
 C\int(1+t)^{  -1}  Z^2dx +
    C|\delta_0|^2 \int  (1+t)^{  -3} V_x^2 dx
        +
         C|\delta_0|^2
         (1+t)^{-\frac{\lambda}{2}-\frac{7}{2}}\\
\leq
 &
     C\int(1+t)^{  -1}  Z^2dx
     +
         C(1+t)^{-\frac{\lambda}{2}-\frac{7}{2}}
         ( \|V_0\|_{H^1}^2+|\delta_0|^2).
  \end{split}
\end{equation}
Integrating  \eqref{a3.54}$\times(1+t)^{2}$,  and using Lemma \ref{Lemma a1},    we get
\begin{equation}\notag
\begin{split}
  (1+t)^{2}\|Z\|^2
        + \int_0^t  (1+s)^{\lambda+2}\|Z_{x}\|^2 ds
       \leq
   C
   ( \|V_0\|_{H^1}^2+\|Z_0 \|^2+|\delta_0|^2).
\end{split}
\end{equation}

Similarly, multiplying \eqref{a3.52} by $(1+t)^{2\lambda} Z_t$, integrating
the resulting equality in $x$ over  $\mathbb{R}$, using Lemmas
\ref{Lemma a1}-\ref{Lemma a3},  the {\it a priori} assumption \eqref{axian}, and after tedious calculations, we have
\begin{equation}\label{a3.53}
\begin{split}
&
-\frac{1}{2}\frac{\rm d}{{\rm d}t} \int
(1+t)^{2\lambda}p'(V_x+\tilde v)Z_x^2dx
  + \frac{\alpha}{2}\int(1+t)^{\lambda} Z_t^2dx
  \\
\leq&
    C \int  (1+t)^{\lambda-2} Z^2 dx +
    C \int  (1+t)^{2\lambda-1} Z_x^2 dx
   +
    C|\delta_0|^2 \int  (1+t)^{\lambda-4} V_x^2 dx\\
   &+
    C|\delta_0|^2 \int  (1+t)^{2\lambda-3} V_{xx}^2 dx
        +
         C|\delta_0|^2(1+t)^{\frac{\lambda}{2}-\frac{9}{2}}
  \\
\leq
        &C \int  (1+t)^{\lambda-2} Z^2 dx +
    C \int  (1+t)^{2\lambda-1} Z_x^2 dx
        +
         C(1+t)^{\frac{\lambda}{2}-\frac{9}{2}}( \|V_0\|_{H^2}^2+\|Z_0 \|_{H^1}^2+|\delta_0|^2).
\end{split}
\end{equation}
Integrating  \eqref{a3.53}$\times(1+t)^{3-\lambda}$,  and using Lemma \ref{Lemma a1}-\ref{Lemma a3}, we  complete the proof of
Lemma \ref{Lemma a5}.
\end{proof}

Similar to the proof of Lemma \ref{Lemma a5}, performing $(\ref{a1.20})_{tx}\times (1+t)^{\lambda} Z_{x}$, $(\ref{a1.20})_{tx}\times (1+t)^{2\lambda}Z_{xt}$ integrating the
resulting  equations with respect to $x$ over $\mathbb{R} $, we can get the
following estimate.
\begin{lemma}\label{Lemma a6}
Under the assumptions of Theorem \ref{Thm a1},  if $\epsilon, |\delta_0|$ are small, it holds that
\begin{equation}\notag
\begin{split}
 &
        (1+t)^{2\lambda+4}\|Z_{xx}\|^2+ \int_0^t \Big((1+s)^{\lambda+4}\|Z_{xt}\|^2+(1+s)^{2\lambda+3}\|Z_{xx}\|^2\Big)ds\\
       \leq   &
   C
   ( \|V_0\|_{H^3}^2+\|Z_0 \|_{H^2}^2+|\delta_0|^2).
\end{split}
\end{equation}
\end{lemma}

\begin{lemma}\label{Lemma a7}
Under the assumptions of Theorem \ref{Thm a1},  if $\epsilon, |\delta_0|$ are small, it holds that
\begin{equation}\notag
\begin{split}
 & (1+t)^{4}\|Z_t\|^2
        +
        (1+t)^{\lambda+5}\|Z_{xt}\|^2+ \int_0^t  \Big( (1+s)^{\lambda+4}\|Z_{xt}\|^2+(1+s)^{5}\|Z_{tt}\|^2\Big)ds\\
       \leq   &
   C
   ( \|V_0\|_{H^3}^2+\|Z_0 \|_{H^2}^2+|\delta_0|^2).
\end{split}
\end{equation}
\end{lemma}
\begin{proof}
Differentiating \eqref{a3.52} with respect to $t$, we get
\begin{equation}\label{a3.70}
(p'(\tilde{v})V_x)_{xtt}
           +\displaystyle
              \frac{\alpha}{(1+t)^\lambda} Z_{tt}+\frac{\alpha \lambda (\lambda+1)}{(1+t)^{\lambda+2}} Z
              =
                \frac{2\alpha\lambda }{(1+t)^{\lambda+1}}Z_t-F_{1tt}.
\end{equation}
 Multiplying \eqref{a3.70} by $(1+t)^{\lambda}  Z_t$, from Lemma
\ref{Lemma a5} and after complicated calculations, we have
\begin{equation}\label{a3.54-}
\begin{split}
&
\frac{\rm d}{{\rm d}t}
\int\bigg(
   \frac{\alpha}{2}  Z_t^2
   +
   \frac{\alpha \lambda(\lambda+1)}{2}(1+t)^{  -2}Z^2\bigg)dx
  +
  \frac{C_0}{2}\int(1+t)^{\lambda} Z_{xt}^2dx
  \\
\leq
 &
 C\int(1+t)^{  -1}  Z_t^2dx + C\int (1+t)^{  -3}  Z^2dx
   + C|\delta_0|^2 \int  (1+t)^{  -5} V_x^2 dx\\
   &+ C|\delta_0|^2 \int  (1+t)^{\lambda-2} Z_{x}^2 dx
        +
         C|\delta_0|^2(1+t)^{-\frac{\lambda}{2}-\frac{11}{2}}\\
\leq
 &
     C\int(1+t)^{  -1}  Z_t^2dx+C\int (1+t)^{  -3}  Z^2dx+ C|\delta_0|^2 \int  (1+t)^{\lambda-2} Z_{x}^2 dx\\
     &
     +
         C(1+t)^{-\frac{\lambda}{2}-\frac{11}{2}}
         ( \|V_0\|_{H^2}^2+\|Z_0\|_{H^1}^2+|\delta_0|^2).
  \end{split}
\end{equation}
Integrating  \eqref{a3.54-}$\times(1+t)^{4}$,  and using Lemmas \ref{Lemma a5}-\ref{Lemma a6},  we get
\begin{equation}\notag
\begin{split}
 (1+t)^{4}\|Z_t\|^2
        + \int_0^t  (1+s)^{\lambda+4}\|Z_{xt}\|^2ds
       \leq
   C
   ( \|V_0\|_{H^3}^2+\|Z_0 \|_{H^2}^2+|\delta_0|^2).
   \end{split}
\end{equation}
Similarly, multiplying \eqref{a3.70} by $(1+t)^{2\lambda} Z_{tt}$, integrating
the resulting equality in $x$ over  $\mathbb{R}$, using Lemmas
\ref{Lemma a1}-\ref{Lemma a6}, and after tedious calculations, we have
\begin{equation}\label{a3.59}
\begin{split}
&
-\frac{1}{2}\frac{\rm d}{{\rm d}t} \int
(1+t)^{2\lambda}p'(V_x+\tilde v)Z_{xt}^2dx
  + \frac{\alpha}{2}\int(1+t)^{\lambda} Z_{tt}^2dx
  \\
\leq
&
    C \int  (1+t)^{\lambda-4} Z^2 dx+C \int  (1+t)^{\lambda-2} Z_t^2 dx
    +
    C \int  (1+t)^{2\lambda-1} Z_{xt}^2 dx  \\
    &
    +
    C|\delta_0|^2 \int  (1+t)^{\lambda-6} V_{x}^2 dx+
    C|\delta_0|^2 \int  (1+t)^{\lambda-4} V_{xt}^2 dx
+
    C|\delta_0|^2 \int  (1+t)^{2\lambda -3} Z_{xx}^2 dx\\
    &
    +C|\delta_0|^2 \int  (1+t)^{2\lambda -5} V_{xx}^2 dx
        +
         C|\delta_0|^2(1+t)^{\frac{\lambda}{2}-\frac{13}{2}}
  \\
\leq
        &  C \int  (1+t)^{\lambda-4} Z^2 dx+C \int  (1+t)^{\lambda-2} Z_t^2 dx
    +
    C \int  (1+t)^{2\lambda-1} Z_{xt}^2 dx\\
        &+
         C(1+t)^{\frac{\lambda}{2}-\frac{13}{2}}
         ( \|V_0\|_{H^2}^2+\|Z_0 \|_{H^1}^2+|\delta_0|^2).
\end{split}
\end{equation}
Integrating  \eqref{a3.59}$\times(1+t)^{5-\lambda}$,  and using Lemmas \ref{Lemma a2}-\ref{Lemma a6}, we complete the proof of
Lemma \ref{Lemma a7}.
\end{proof}

With Lemmas \ref{Lemma a1}-\ref{Lemma a7} in hand, by the Sobolev inequality and $0\leq \lambda<1$, it's easy to  know that
\begin{equation}\notag
\begin{split}
       \| V_x(\cdot,t)\|_{L^\infty}
       &\leq
   C(1+t)^{-\frac{3(\lambda+1)}{4}}( \|V_0\|_{H^3}+\|Z_0 \|_{H^2}+|\delta_0|)
\leq \frac{\epsilon}{2},
   \end{split}
\end{equation}
  \begin{equation}\notag
  \begin{split}
       \| V_{xt}(\cdot,t)\|_{L^\infty}&\leq
   C(1+t)^{-\frac{3\lambda+7}{4}}( \|V_0\|_{H^3}+\|Z_0 \|_{H^2}+|\delta_0|  )
 \leq \frac{\epsilon}{2}(1+t)^{-1},
   \end{split}
\end{equation}
and
  \begin{equation}\notag
  \begin{split}
       \| V_{xx}(\cdot,t)\|_{L^\infty}&\leq
   C(1+t)^{-\frac{5\lambda+5}{4}}( \|V_0\|_{H^3}+\|Z_0 \|_{H^2}+|\delta_0| )
 \leq \frac{\epsilon}{2}(1+t)^{-\frac{\lambda+1}{2}},
   \end{split}
\end{equation}
provided $\|V_0\|_{H^3}+\|Z_0 \|_{H^2}+|\delta_0| \ll 1$. Up to now, we thus close the {\it a priori} assumption \eqref{axian}  about $(V_x, V_{xt}, V_{xx})$ from  Lemmas \ref{Lemma a1}-\ref{Lemma a6}. The proof of Theorem \ref{Thm a1} is completed.

Furthermore, if $V_0\in H^5(\mathbb{R})$, $Z_0\in H^4(\mathbb{R})$, we can get the following lemmas.

\begin{lemma}\label{Lemma a8}
Under the assumptions of Theorem \ref{Thm a1}, if $\epsilon, |\delta_0|$ are small, it holds that
\begin{equation}\notag%\label{a3.31}
\begin{split}
(1+t)^{ 4\lambda+4}\|\partial_x^4V \|^2+ \int_0^t \Big((1+s)^{3\lambda+4}\|\partial_x^3Z\|^2+(1+s) ^{4\lambda+3} \|\partial_x^4 V \|^2)\Big) ds
       \leq
   C ( \|V_0\|_{H^4}^2+|\delta_0|^2).
\end{split}
\end{equation}
\end{lemma}

\begin{lemma}\label{Lemma a9}
Under the assumptions of Theorem \ref{Thm a1}, if $\epsilon, |\delta_0|$ are small, it holds that
\begin{equation}\notag%\label{a3.31}
\begin{split}
(1+t)^{ 5\lambda+5}\|\partial_x^5V \|^2+ \int_0^t \Big((1+s)^{4\lambda+5}\|\partial_x^4Z\|^2+(1+s) ^{5\lambda+4} \|\partial_x^5 V \|^2)\Big) ds
       \leq
   C ( \|V_0\|_{H^5}^2+|\delta_0|^2).
\end{split}
\end{equation}
\end{lemma}

\begin{lemma}\label{Lemma a10}
Under the assumptions of Theorem \ref{Thm a1}, if $\epsilon, |\delta_0|$ are small, it holds that
\begin{equation}\notag
\begin{split}
 &
        (1+t)^{3\lambda+5}\|\partial^3_x Z \|^2+ \int_0^t \Big((1+s)^{2\lambda+5}\|Z_{xxt}\|^2+(1+s)^{3\lambda+4}\|\partial^3_x Z\|^2\Big)ds\\
       \leq   &
   C
   ( \|V_0\|_{H^4}^2+\|Z_0 \|_{H^3}^2+|\delta_0|^2).
\end{split}
\end{equation}

\end{lemma}

\begin{lemma}\label{Lemma a11}
Under the assumptions of Theorem \ref{Thm a1}, if $\epsilon, |\delta_0|$ are small, it holds that
\begin{equation}\notag
\begin{split}
 &
        (1+t)^{4\lambda+6}\|\partial^4_x Z \|^2+ \int_0^t \Big((1+s)^{3\lambda+6}\|\partial_x^3Z_{t}\|^2+(1+s)^{4\lambda+5}\|\partial^4_x Z\|^2\Big)ds\\
       \leq   &
   C
   ( \|V_0\|_{H^5}^2+\|Z_0 \|_{H^4}^2+|\delta_0|^2).
\end{split}
\end{equation}
\end{lemma}
\begin{lemma}\label{Lemma a12}
Under the assumptions of Theorem \ref{Thm a1},  if $\epsilon, |\delta_0|$ are small, it holds that
\begin{equation}\notag
\begin{split}
        (1+t)^{2\lambda+6}\|Z_{xxt}  \|^2+(1+t)^{3\lambda+7}\|\partial^3_xZ_{t} \|^2
       \leq
   C
   ( \|V_0\|_{H^5}^2+\|Z_0 \|_{H^4}^2+|\delta_0|^2).
\end{split}
\end{equation}

\end{lemma}
Restricting the solution $(\bar v, \bar u)(x,t)$ of the Cauchy problem \eqref{a1.7} on  $\mathbb{R}^+$, we get the solution of the Dirichlet initial-boundary value problem \eqref{a1.5}. From \eqref{a1.14}, Lemmas \ref{Lemma a0.1}, \ref{Thm a1}, \ref{Lemma a8}-\ref{Lemma a12} and Minkowski's inequality $\|u+v\|_{L^p}\leq \|u\|_{L^p}+\|v\|_{L^p}$, we have

\begin{proposition}{\bf (Decay rates of the nonlinear diffusion waves).} \label{Prop a1}
For any $\alpha>0$, $(V_0,Z_0)(x)\in H^5(\mathbb{R})\times H^4(\mathbb{R})$, assume that
  $\| v_0-v_+\|_{L^1(\mathbb{R}^+)} +|u_+|+\|V_0\|_{H^3(\mathbb{R})}+\|Z_0\|_{H^2(\mathbb{R})}$ is sufficiently small. Then, for any $0\leq\lambda<1$, let  $\delta=\|   v_0-v_+\|_{L^1(\mathbb{R}^+)} +|u_+|+\|V_0\|_{H^5(\mathbb{R})}+\|Z_0\|_{H^4(\mathbb{R})}$,
there exists a unique time-global solution $(\bar v, \bar u)(x,t)$ of the initial-boundary value problem \eqref{a1.5} satisfying
\begin{equation}\notag
\begin{split}
       \|\partial_x^k(\bar v(t)-  v_+)\|_{L^p(\mathbb{R}^+)}
            \leq
                   C\delta  (1+t)^{-\frac{\lambda+1}{2}(1-\frac{1}{p})-\frac{(\lambda+1)k}{2}},
\end{split}
\end{equation}
where  $ k\leq 3$,  {if} $ p\in(2,\infty]$;   $ k\leq 4$, {if} $ p=2,$
\begin{equation}\notag
\begin{split}
       \|\partial_t\partial_x^k\bar v(t)\|_{L^p(\mathbb{R}^+)}
            \leq
                   C\delta  (1+t)^{-\frac{\lambda+1}{2}(1-\frac{1}{p})-\frac{(\lambda+1)k}{2}-1},
\end{split}
\end{equation}
where  $k\leq 2$,  {if} $ p\in(2,\infty]$;   $k\leq 3$, {if} $ p=2,$
\begin{equation}\notag
\begin{split}
       \|\partial_t^2\partial_x^k\bar v(t)\|_{L^p(\mathbb{R}^+)}
            \leq
                   C\delta  (1+t)^{-\frac{\lambda+1}{2}(1-\frac{1}{p})-\frac{(\lambda+1)k}{2}-2},
\end{split}
\end{equation}
where $k\leq 1$, if $p\in(2,\infty]$;  $k\leq 2$, if $ p=2$.
\end{proposition}

Finally, we give the following dissipative property of the correction function $\hat v(x,t)$ as in \cite{Cui-Yin-Zhang-Zhu2016}.
\begin{proposition}\label{Prop a2}
For any $0\leq \lambda <1$, then there exist constants $0<\vartheta<1-\lambda$, $c$ and $C>0$, the correction function $\hat v(x,t)$
defined in \eqref{1.10} satisfies the following dissipative
estimates:
\begin{equation}\notag
\|\hat v\|_{L^1(\mathbb{R}^+)}+\|\hat v\|_{H^\infty(\mathbb{R}^+)}+\|\hat v_t\|_{H^\infty(\mathbb{R}^+)} \leq C|u_+| e^{-ct^\vartheta}.
\end{equation}
\end{proposition}

\subsection{Proofs of Theorems \ref{Thm 1}-\ref{Thm 1-1}}\label{S2.3}
In this subsection, we prove  Theorems \ref{Thm 1}-\ref{Thm 1-1}.
To begin with, we give the local (in time) estimates of the initial boundary problem
\eqref{1.20}-\eqref{1.21} for any $0\leq \lambda<1$. The proofs are quite similar to those in Section 3 of \cite{Cui-Yin-Zhang-Zhu2016}. Thus we omit the details here.
\begin{proposition}{\bf (Locally estimates).}  \label{prop 3} Under the conditions of Theorems \ref{Thm 1}-\ref{Thm 1-1}, for any given $T>0$, $0\leq \lambda<1$, $\alpha>0$, the solution $(\omega, \omega_t)(x,t)$ to the initial boundary problem
\eqref{1.20}-\eqref{1.21} on $[0,T]$ satisfying
\begin{equation}\notag
  \begin{split}
  &
  \|\omega\|^2+ \|\omega_t\|^2+\|\omega_x\|^2
  +
  \int_0^t \big(\|\omega_t\|^2+\|\omega_x\|^2\big)ds
  \leq
  C(T)(\|\omega_0\|_{H^1}^2+\|z_0 \|^2+\delta),
  \end{split}
\end{equation}
\begin{equation}\notag
 \begin{split}
 &
 \|\omega_{xx}\|^2+\|\omega_{xt}\|^2
 +
 \int_0^t\big( \|\omega_{xx}\|^2+ \|\omega_{xt}\|^2\big)ds
 \leq
 C(T)( \|\omega_0\|_{H^2}^2+\|z_0 \|_{H^1}^2+\delta),
 \end{split}
 \end{equation}
 and
 \begin{equation}\notag
 \begin{split}
 &
 \|\omega_{xxx}\|^2+\|\omega_{xxt}\|^2
 +
 \int_0^t\big(\|\omega_{xxx}\|^2+\|\omega_{xxt}\|^2\big)ds
 \leq
 C(T)( \|\omega_0\|_{H^3}^2+\|z_0 \|_{H^2}^2+\delta).
 \end{split}
 \end{equation}
\end{proposition}

Now we devote ourselves to the estimates of the global solution $(\omega,z)(x,t)$ under the {\it a priori} assumption
\begin{equation}\label{xian}
N(T):=\sup_{0<t<T} \bigg\{\| \omega_x(\cdot,t)\|_{L^\infty}+ (1+t)\| \omega_{xt}(\cdot,t)\|_{L^\infty}+ (1+t)^{\frac{\lambda+1}{2}}\| \omega_{xx}(\cdot,t)\|_{L^\infty}\bigg\}\leq \epsilon,
\end{equation}
for some $0<\epsilon\ll 1$ and $0<T<\infty$. Throughout this subsection all estimates are independent  of $T$.

It can be checked  that
\begin{equation}\label{bou}
\left\{
\begin{split}
&
 \omega(0,t)=\omega_{xx}(0,t)=\omega_{t}(0,t)=\omega_{txx}(0,t)
 =0,
 \qquad\quad \qquad \mbox{etc},\\
 & \omega(\infty,t)=\omega_{x}(\infty,t)=\omega_{t}(\infty,t)=\omega_{xx}(\infty,t)
 =0,
 \qquad \qquad \mbox{etc}.
 \end{split} \right.
\end{equation}
We need  the following lemma concerning the lower order  estimates on $(\omega,\omega_t)$.
\begin{lemma}\label{Lemma 1}
Under the assumptions of Theorems \ref{Thm 1}-\ref{Thm 1-1}, if $\epsilon, \delta$ are small, it holds that
\begin{equation}\label{3.1a}
\begin{split}
      &
       \|\omega\|^2+(1+t)^{2\lambda}(\|\omega_t\|^2+\|\omega_x\|^2)
       +
       \int_0^t(1+s)^{\lambda}\big(\|\omega_t\|^2+\|\omega_x\|^2\big)ds\\
\leq
 &
   C( \|\omega_0\|_{H^1}^2+\|z_0 \|^2+\delta),  \qquad \mbox{for any} \qquad 0\leq\lambda<\frac{3}{5},
\end{split}
\end{equation}
and for $\frac{3}{5}<\lambda<1$
\begin{equation}\label{3.1b}
\left\{
\begin{split}
&      (1+t)^{\frac{3}{2}-\frac{5\lambda}{2}} \|\omega\|^2
       +
       (1+t)^{\frac{3}{2}-\frac{\lambda}{2}}(\|\omega_x\|^2+\|\omega_t\|^2)
       \leq
   C( \|V_0\|_{H^1}^2+\|z_0 \|^2+\delta),\\
&
       \int_0^t\bigg[(1+s)^{\beta-\lambda-1}\|\omega\|^2+(1+s)^{\beta}\big(\|\omega_x\|^2+\|\omega_t\|^2\big)
       \bigg]ds\\
&\qquad \leq
   C(1+t)^{\beta+\frac{3\lambda}{2}-\frac{3}{2}}( \|\omega_0\|_{H^1}^2+\|z_0 \|^2+\delta),\quad  \mbox{for any }\quad  \frac{3}{2}-\frac{3\lambda}{2}<\beta<\lambda.
\end{split} \right.
\end{equation}
\end{lemma}
\begin{proof}
Multiplying \eqref{1.20} by $(1+t)^\beta \omega$, integrating the
resulting  equality with respect to $x$ over $\mathbb{R}^+ $, using the boundary condition  \eqref{bou}, one obtains
\begin{equation}\label{3.3}
\begin{split}
&
\frac{\rm d}{{\rm d}t} \int_0^\infty
 \left[
   (1+t)^{\beta}\omega \omega_t
    +
   \frac{\alpha}{2}(1+t)^{\beta-\lambda}\omega^2
 \right]dx
  - \int_0^\infty(1+t)^{\beta} p'(\bar v)\omega_x^2dx\\&+
     \frac{\alpha(\lambda-\beta)}{2}\int_0^\infty(1+t)^{\beta-\lambda-1}\omega^2dx
  \\
=
 &
  \int_0^\infty(1+t)^{\beta} \omega_t^2dx+\frac{\rm d}{{\rm d}t} \int_0^\infty
   \frac{\beta}{2}(1+t)^{\beta-1}\omega^2dx
    +
     \frac{\beta(1-\beta)}{2}\int_0^\infty(1+t)^{\beta-2}\omega^2dx\\
     &
     +\int_0^\infty (1+t)^{\beta} F\omega dx.
\end{split}
\end{equation}
Now we estimate the last term in the  right hand of \eqref{3.3} as follows:
\begin{equation}\label{3.4}
\begin{split}
\int_0^\infty (1+t)^{\beta} F\omega dx
=\int_0^\infty &\bigg[\frac{1}{\alpha}(1+t)^{\beta+\lambda} p(\bar v)_{xt}
      +
       \frac{\lambda}{\alpha} { (1+t)^{\beta+\lambda-1} } { p(\bar v)_x}
       \\
         & -
            (1+t)^\beta\big(p(\omega_x+\bar v+\hat v)-p(\bar{v})-p'(\bar v)\omega_x\big)_x\bigg]\omega dx.
\end{split}
\end{equation}
Firstly, from Proposition \ref{Prop a1}  and \eqref{bou},   we have
\begin{equation}\label{3.5}
\begin{split}
\int_0^\infty \frac{1}{\alpha}(1+t)^{\beta+\lambda} p(\bar v)_{xt}\omega dx
=&-\int_0^\infty \frac{1}{\alpha}(1+t)^{\beta+\lambda} p(\bar v)_{t}\omega_xdx\\
\leq& \frac{C_0}{8}\int_0^\infty(1+t)^{\beta}  \omega_x^2dx+C\int_0^\infty
(1+t)^{\beta+2\lambda}|\bar v_t|^2dx \\  \leq&
\frac{C_0}{8}\int_0^\infty(1+t)^{\beta}
\omega_x^2dx+C\delta^2(1+t)^{\beta+\frac{3\lambda}{2}-\frac{5}{2}},
\end{split}
\end{equation}
and for some constants $\kappa>1,\nu>0$ which will be determined below
\begin{equation}\label{3.6}
\begin{split}
  \int_0^\infty \frac{\lambda}{\alpha} { (1+t)^{\beta+\lambda-1} } { p(\bar
  v)_x}\omega dx
\leq& \nu\int_0^\infty(1+t)^{-\kappa}  \omega^2dx
    +
      C\int_0^\infty (1+t)^{2\beta+2\lambda+\kappa-2}|\bar v_x|^2dx \\
       \leq&
         \nu\int_0^\infty(1+t)^{-\kappa} \omega^2dx
             +\frac{C\delta^2}{\nu}(1+t)^{2\beta+\frac{\lambda}{2}+\kappa-\frac{7}{2}}.
\end{split}
\end{equation}
Secondly, by using Propositions \ref{Prop a1}-\ref{Prop a2}, \eqref{bou} and the {\it a priori} assumption \eqref{xian}, we have
\begin{equation}\label{3.7}
\begin{split}
&
-\int_0^\infty (1+t)^\beta\big(p(\omega_x+\bar v+\hat v)-p(\bar{v})-p'(\bar v)\omega_x\big)_x\omega dx\\
   =
    &
   \int_0^\infty (1+t)^\beta\big(p(\omega_x+\bar v+\hat v)-p(\bar{v})-p'(\bar v)\omega_x\big)\omega_xdx\\
    \leq
       &
       \frac{C_0}{8}\int_0^\infty (1+t)^{\beta}  \omega_x^2dx
       +
       C\int_0^\infty (1+t)^{\beta}(|\hat v|^2+|\omega_x|^4)dx \\
          \leq
           &
             \frac{C_0}{8}
             \int_0^\infty (1+t)^{\beta}
             \omega_x^2dx
             +
             C\epsilon^2\int_0^\infty (1+t)^{\beta}  \omega_x^2dx
             +
             C\delta^2(1+t)^{\beta+\frac{3\lambda}{2}-\frac{5}{2}}.
\end{split}
\end{equation}
Substituting \eqref{3.4}-\eqref{3.7} into  \eqref{3.3}, and using the smallness of $\epsilon$, we have
\begin{equation}\label{3.8}
\begin{split}
&
\frac{\rm d}{{\rm d}t} \int_0^\infty
 \left[
   (1+t)^{\beta}\omega \omega_t
   +
   \frac{\alpha}{2}(1+t)^{\beta-\lambda}\omega^2
 \right]dx
  +\frac{C_0}{2}\int_0^\infty (1+t)^{\beta}\omega_x^2dx\\&+
     \frac{\alpha(\lambda-\beta)}{2}\int_0^\infty(1+t)^{\beta-\lambda-1}\omega^2dx
  \\
\leq
 &
  \int_0^\infty  (1+t)^{\beta} \omega_t^2 dx+  \frac{\rm d}{{\rm d}t} \int_0^\infty
   \frac{\beta}{2}(1+t)^{\beta-1}\omega^2dx
    +
     \frac{\beta(1-\beta)}{2}\int_0^\infty(1+t)^{\beta-2}  \omega^2dx\\
     &
       +
         \nu\int_0^\infty(1+t)^{-\kappa} \omega^2dx +
         C\delta^2(1+t)^{\beta+\frac{3\lambda}{2}-\frac{5}{2}}
             +\frac{C\delta^2}{\nu}(1+t)^{2\beta+\frac{\lambda}{2}+\kappa-\frac{7}{2}}.
\end{split}
\end{equation}
Multiplying \eqref{1.20} by $(1+t)^{\beta+\lambda} \omega_t$, and
integrating the resulting equality over $\mathbb{R}^+$, we see that
\begin{equation}\label{3.9}
\begin{split}
&
\frac{1}{2}\frac{\rm d}{{\rm d}t} \int_0^\infty
 \left[
   (1+t)^{\beta+\lambda}\omega_t^2-(1+t)^{\beta+\lambda}p'(\bar v)\omega_x^2
 \right]dx
  + \int_0^\infty\alpha(1+t)^{\beta}\omega_t^2dx
  \\
=
 &
  -
    \frac{1}{2}\int_0^\infty (1+t)^{\beta+\lambda} p''(\bar v)\bar v_t\omega_x^2dx
    -
     \frac{\beta+\lambda}{2}\int_0^\infty(1+t)^{\beta+\lambda-1}p'(\bar v)\omega_x^2 dx\\
     &+\frac{\beta+\lambda }{2}\int_0^\infty(1+t)^{\beta+\lambda-1}\omega_t^2dx
     +
       \int_0^\infty (1+t)^{\beta+\lambda} F\omega_tdx\\
        \leq
           &
            C\int_0^\infty(1+t)^{\beta+\lambda-1}(\omega_x^2+\omega_t^2) dx
     +
       \int_0^\infty (1+t)^{\beta+\lambda} F\omega_tdx.
\end{split}
\end{equation}
Now we estimate the last term in the  right hand of \eqref{3.9} as follows:
\begin{equation}\label{3.10}
\begin{split}
\int_0^\infty (1+t)^{\beta+\lambda} F\omega_tdx = & \int_0^\infty
\bigg[\frac{1}{\alpha}(1+t)^{\beta+2\lambda} p(\bar v)_{xt}
      +
       \frac{\lambda}{\alpha} { (1+t)^{\beta+2\lambda-1} } { p(\bar v)_x}\\
          &\qquad -
            (1+t)^{\beta+\lambda}\big(p(\omega_x+\bar v+\hat v)-p(\bar{v})-p'(\bar v)\omega_x\big)_x\bigg]\omega_tdx.
\end{split}
\end{equation}
Firstly, applying Proposition \ref{Prop a1}, one gets
\begin{equation}\label{3.11}
\begin{split}
&
 \int_0^\infty \frac{1}{\alpha}(1+t)^{\beta+2\lambda} p(\bar v)_{xt}\omega_tdx
   \\
    \leq
     &
       \frac{\alpha}{4}\int_0^\infty(1+t)^{\beta}  \omega_t^2dx
       +
       C\int_0^\infty (1+t)^{\beta+4\lambda}(|\bar v_{xt}|^2+|\bar v_x|^2 |\bar v_t|^2)dx \\
          \leq
            &
            \frac{\alpha}{4}\int_0^\infty(1+t)^{\beta}  \omega_t^2dx
            +
            C\delta^2(1+t)^{\beta+\frac{5\lambda}{2}-\frac{7}{2}},
\end{split}
\end{equation}
and
\begin{equation}\label{3.12}
\begin{split}
&\int_0^\infty \frac{\lambda}{\alpha} { (1+t)^{\beta+2\lambda-1} } { p(\bar
v)_x}\omega_tdx \\
\leq&
      \frac{\alpha}{4}\int_0^\infty (1+t)^{\beta}  \omega_t^2dx
      +
      C\int_0^\infty (1+t)^{\beta+4\lambda-2}|\bar v_x|^2dx \\
\leq&
            \frac{\alpha}{4}\int_0^\infty(1+t)^{\beta}  \omega_t^2dx+C\delta^2(1+t)^{\beta+\frac{5\lambda}{2}-\frac{7}{2}}.
\end{split}
\end{equation}
Next by using Propositions \ref{Prop a1}-\ref{Prop a2} and the {\it a
priori} assumption \eqref{xian}, we have
\begin{equation}\label{3.13}
\begin{split}
   &
    -\int_0^\infty (1+t)^{\beta+\lambda}\big(p(\omega_x+\bar v+\hat v)-p(\bar{v})-p'(\bar v)\omega_x\big)_x\omega_tdx\\
      =
       &
        \int_0^\infty (1+t)^{\beta+\lambda}
        \big(
        p(\omega_x+\bar v+\hat v)-p(\bar{v})-p'(\bar v)\omega_x
        \big)
        \omega_{xt}dx
        \\
         =
         &
           (1+t)^{\beta+\lambda}\frac{\rm d}{{\rm d}t}\int_0^\infty
           \bigg(
           \int_{\bar v}^{\omega_x+\bar v+\hat v} p(s)ds
           -
           p(\bar v)\omega_x
           -
           \frac{p'(\bar v)}{2}\omega_x^2 \bigg)dx\\
           &
           -
            \int_0^\infty(1+t)^{\beta+\lambda} \bar v_t\Big(p(\omega_x+\bar v+\hat v)-p(\bar{v})-p'(\bar v)\omega_x-\frac{p''(\bar v)}{2}\omega_x^2\Big)dx\\
            &
            -
              \int_0^\infty(1+t)^{\beta+\lambda} p(\omega_x+\bar v+\hat v)\hat v_t dx\\
              \leq
              &
                \frac{\rm d}{{\rm d}t}
                \int_0^\infty
                (1+t)^{\beta+\lambda}\bigg(
                \int_{\bar v}^{\omega_x+\bar v+\hat v} p(s)ds
                -
                p(\bar v)\omega_x
                -
                \frac{p'(\bar v)}{2}\omega_x^2
                \bigg)dx \\
        &
        -
          (\beta+\lambda)\int_0^\infty (1+t)^{\beta+\lambda-1}  \bigg(  \int_{\bar v}^{\omega_x+\bar v+\hat v} p(s)ds-p(\bar v)\omega_x-\frac{p'(\bar v)}{2}\omega_x^2 \bigg)dx\\
          &
           +C\delta\int_0^\infty (1+t)^{\beta+\lambda-1}(|\hat v|+|\omega_x|^3)dx
           + C\int_0^\infty (1+t)^{\beta+\lambda} |\hat v_t| dx \\
            \leq
              &
               \frac{\rm d}{{\rm d}t}\int_0^\infty (1+t)^{\beta+\lambda}\bigg(  \int_{\bar v}^{\omega_x+\bar v+\hat v} p(s)ds-p(\bar v)\omega_x-\frac{p'(\bar v)}{2}\omega_x^2 \bigg)dx\\
               &
               +C(\delta+\epsilon)\int_0^\infty(1+t)^{\beta+\lambda-1}  \omega_x^2dx+C\delta(1+t)^{\beta+\frac{5\lambda}{2}-\frac{7}{2}}.
\end{split}
\end{equation}
Putting  \eqref{3.10}-\eqref{3.13} into \eqref{3.9} implys
\begin{equation}\label{3.14}
\begin{split}
&
\frac{1}{2}\frac{\rm d}{{\rm d}t} \int_0^\infty
 \left[
   (1+t)^{\beta+\lambda}\omega_t^2-(1+t)^{\beta+\lambda}p'(\bar v)\omega_x^2
 \right]dx
  + \frac{\alpha}{2}\int_0^\infty (1+t)^{\beta}\omega_t^2dx
  \\
\leq&
     \frac{\rm d}{{\rm d}t}\int_0^\infty (1+t)^{\beta+\lambda}\bigg(  \int_{\bar v}^{\omega_x+\bar v+\hat v} p(s)ds-p(\bar v)\omega_x-\frac{p'(\bar v)}{2}\omega_x^2 \bigg)dx\\
   &
   +
    C \int_0^\infty  (1+t)^{\beta+\lambda-1} (\omega_x^2+\omega_t^2) dx
        +
         C\delta(1+t)^{\beta+\frac{5\lambda}{2}-\frac{7}{2}}.
\end{split}
\end{equation}
Multiplying \eqref{3.14} by $h$, and adding up the resulting
inequality and \eqref{3.8}, we get
\begin{equation}\label{3.15}
\begin{split}
&
\frac{\rm d}{{\rm d}t} \int_0^\infty
 \bigg[(1+t)^{\beta}\omega\omega_t
 +
  \frac{\alpha}{2}(1+t)^{\beta-\lambda}\omega^2
 +
 \frac{h}{2}
   (1+t)^{\beta+\lambda}\omega_t^2-\frac{h}{2}(1+t)^{\beta+\lambda}p'(\bar v)\omega_x^2
 \bigg]dx\\
  &
    +
    \frac{C_0}{2}\int_0^\infty (1+t)^{\beta}\omega_x^2dx
      +
        \left( \frac{\alpha h}{2}-1\right)\int_0^\infty (1+t)^{\beta}\omega_t^2dx+
     \frac{\alpha(\lambda-\beta)}{2}\int_0^\infty(1+t)^{\beta-\lambda-1}\omega^2dx
         \\
        \leq
        &
          h\frac{\rm d}{{\rm d}t}\int_0^\infty (1+t)^{\beta+\lambda}\bigg(  \int_{\bar v}^{\omega_x+\bar v+\hat v} p(s)ds-p(\bar v)\omega_x-\frac{p'(\bar v)}{2}\omega_x^2 \bigg)dx + \frac{\beta}{2}\frac{\rm d}{{\rm d}t} \int_0^\infty
(1+t)^{\beta-1}\omega^2dx\\
     &
    +
     \nu\int_0^\infty(1+t)^{-\kappa}  \omega^2dx+
      Ch \int_0^\infty  (1+t)^{\beta+\lambda-1} (\omega_x^2+\omega_t^2) dx+
     \frac{\beta(1-\beta)}{2}\int_0^\infty(1+t)^{\beta-2}  \omega^2dx\\
     &
          +
         C\delta^2(1+t)^{\beta+\frac{3\lambda}{2}-\frac{5}{2}}
             +\frac{C\delta^2}{\nu}(1+t)^{2\beta+\frac{\lambda}{2}+\kappa-\frac{7}{2}}.
\end{split}
\end{equation}
{\bf Case 1. $0\leq\lambda<\frac{3}{5}$}

It is easy to know that $1<\frac{5}{2}-\frac{5\lambda}{2}$. Therefore we can take $\beta=\lambda$, $\nu=1$, and there exists constant $\kappa$ satisfying $1<\kappa<\frac{5}{2}-\frac{5\lambda}{2}$. Then we have
\begin{align}\label{3.15c1}
&
\frac{\rm d}{{\rm d}t} \int_0^\infty
 \bigg[(1+t)^{\lambda}\omega\omega_t
 +
  \frac{\alpha}{2} \omega^2
 +
 \frac{h}{2}
   (1+t)^{2\lambda}\omega_t^2-\frac{h}{2}(1+t)^{2\lambda}p'(\bar v)\omega_x^2
 \bigg]dx\nonumber\\
  &
    +
    \frac{C_0}{2}\int_0^\infty(1+t)^{\lambda}\omega_x^2dx
      +
        \left( \frac{\alpha h}{2}-1\right)\int_0^\infty (1+t)^{\lambda}\omega_t^2dx\nonumber
         \\
        \leq
        &
          h\frac{\rm d}{{\rm d}t}\int_0^\infty (1+t)^{2\lambda}\bigg(  \int_{\bar v}^{\omega_x+\bar v+\hat v} p(s)ds-p(\bar v)\omega_x-\frac{p'(\bar v)}{2}\omega_x^2 \bigg)dx
          +\frac{\lambda}{2}\frac{\rm d}{{\rm d}t} \int_0^\infty(1+t)^{\lambda-1}\omega^2dx\nonumber\\
          &
          +\int_0^\infty(1+t)^{-\kappa}  \omega^2dx
   +
     \frac{\lambda(1-\lambda)}{2}\int_0^\infty(1+t)^{\lambda-2}\omega^2dx
     +
      Ch \int_0^\infty  (1+t)^{2\lambda-1} (\omega_x^2+\omega_t^2) dx\nonumber\\
      &
         +C\delta(1+t)^{\frac{5\lambda}{2}+\kappa-\frac{7}{2}}.
\end{align}
Let  $T_0$ sufficiently large such that
\begin{equation}\notag
\left\{
\begin{array}{l}
Ch(1+t)^{\lambda-1}\leq \frac{C_0}{4},\\[2mm]
Ch(1+t)^{\lambda-1}\leq \frac{1}{2}(\frac{\alpha h}{2}-1),\\[2mm]
\frac{\lambda}{2}(1+t)^{\lambda-1}\leq \frac{1}{4},
\end{array}\right.
\end{equation}
if $t\geq T_0$, and fixing $h=\frac{6}{\alpha}$,  we have
\begin{equation}\notag
\begin{split}
&
\frac{\rm d}{{\rm d}t} \mathcal{H}(t)
    +
    \frac{C_0}{4}\int_0^\infty(1+t)^{\lambda}\omega_x^2dx
      +
        \frac{1}{2}\left(\frac{\alpha h}{2}-1\right)\int_0^\infty (1+t)^{\lambda}\omega_t^2dx
         \\
        \leq
        &
          C\frac{\rm d}{{\rm d}t}\int_0^\infty (1+t)^{2\lambda}\bigg(  \int_{\bar v}^{\omega_x+\bar v+\hat v} p(s)ds-p(\bar v)\omega_x-\frac{p'(\bar v)}{2}\omega_x^2 \bigg)dx
          +\frac{\lambda}{2}\frac{\rm d}{{\rm d}t} \int_0^\infty(1+t)^{\lambda-1}\omega^2dx\\
          &
    +
     C\int_0^\infty(1+t)^{-\kappa}  \omega^2dx+
     \frac{\lambda(1-\lambda)}{2}\int_0^\infty(1+t)^{\lambda-2}\omega^2dx
        +
         C\delta(1+t)^{\frac{5\lambda}{2}-\frac{7}{2}+\kappa},
\end{split}
\end{equation}
where
\begin{equation}\notag
  \mathcal{H}(t)
    \sim
     \|\omega\|^2
      +
       (1+t)^{2\lambda}\|\omega_x\|^2
        +
       (1+t)^{2\lambda}\|\omega_t\|^2.
\end{equation}
Then, we have
\begin{equation}\notag
\begin{split}
&
\frac{\rm d}{{\rm d}t}\mathcal{H}(t)
    +
    \frac{C_0}{4}\int_0^\infty(1+t)^{\lambda}\omega_x^2dx
      +
        \int_0^\infty  (1+t)^{\lambda}\omega_t^2dx
         \\
        \leq
        &
          C\frac{\rm d}{{\rm d}t}\int_0^\infty (1+t)^{2\lambda}\bigg(  \int_{\bar v}^{\omega_x+\bar v+\hat v} p(s)ds-p(\bar v)\omega_x-\frac{p'(\bar v)}{2}\omega_x^2 \bigg)dx+\frac{\lambda}{2}\frac{\rm d}{{\rm d}t} \int_0^\infty(1+t)^{\lambda-1}\omega^2dx\\
     &
    +
     C(1+t)^{-\kappa}  \mathcal{H}(t)+C(1+t)^{\lambda-2}\mathcal{H}(t)
        +
         C\delta(1+t)^{\frac{5\lambda}{2}-\frac{7}{2}+\kappa}, \qquad \mbox{for any}\qquad  t\in [T_0, \infty).
         \end{split}
\end{equation}
Using  Gronwall’s inequality on $[T_0,t]$,  one has by $1<\kappa<\frac{5}{2}-\frac{5\lambda}{2}$ and $0\leq \lambda< \frac{3}{5}$,
\begin{equation}\label{3.18}
\begin{split}
&
 \mathcal{H}(t)
    +
    c\int_{T_0}^t  (1+s)^{\lambda}\big(\|\omega_x\|^2
      +
        \|\omega_t\|^2 \big)ds
         \\
        \leq
        &
          C\int_0^\infty (1+t)^{2\lambda}\bigg(  \int_{\bar v}^{\omega_x+\bar v+\hat v} p(s)ds-p(\bar v)\omega_x-\frac{p'(\bar v)}{2}\omega_x^2 \bigg)dx+\frac{\lambda}{2} \int_0^\infty(1+t)^{\lambda-1}\omega^2dx\\
     &
        +C(\mathcal{H}(T_0)+
         \delta)\\
       \leq
        &
          C\epsilon\int_0^\infty (1+t)^{2\lambda}|\omega_x|^2dx+\frac{1}{4} \int_0^\infty \omega^2dx+C(\mathcal{H}(T_0)+
         \delta),
         \end{split}
\end{equation}
which together with  Proposition \ref{prop 3} deduce \eqref{3.1a} in view of the smallness of $\epsilon$.\\

{\bf Case 2. $\frac{3}{5}<\lambda<1$}

 In this case, from $\lambda>\frac{3}{2}-\frac{3\lambda}{2}$,  we can take  $\frac{3}{2}-\frac{3\lambda}{2}<\beta<\lambda$,
$\kappa=\lambda-\beta+1$, and $\nu= \frac{\alpha(\lambda-\beta)}{4}>0$. Therefore from \eqref{3.15}, we know
\begin{align}\label{3.15c2}
&
\frac{\rm d}{{\rm d}t} \int_0^\infty
 \bigg[(1+t)^{\beta}\omega\omega_t
 +
  \frac{\alpha}{2}(1+t)^{\beta-\lambda}\omega^2
 +
 \frac{h}{2}
   (1+t)^{\beta+\lambda}\omega_t^2-\frac{h}{2}(1+t)^{\beta+\lambda}p'(\bar v)\omega_x^2
 \bigg]dx\nonumber\\
 &+\frac{\alpha(\lambda-\beta)}{4}\int_0^\infty(1+t)^{\beta-\lambda-1}  \omega^2dx
   +
    \frac{C_0}{2}\int_0^\infty(1+t)^{\beta}\omega_x^2dx
      +
        \left(  \frac{\alpha h}{2}-1\right) \int (1+t)^{\beta}\omega_t^2dx\nonumber
         \\
        \leq
        &
          h\frac{\rm d}{{\rm d}t}\int_0^\infty (1+t)^{\beta+\lambda}\bigg(  \int_{\bar v}^{\omega_x+\bar v+\hat v} p(s)ds-p(\bar v)\omega_x-\frac{p'(\bar v)}{2}\omega_x^2 \bigg)dx +\frac{\beta}{2}\frac{\rm d}{{\rm d}t} \int_0^\infty(1+t)^{\beta-1}\omega^2dx\nonumber\\
          &+
     \frac{\beta(1-\beta)}{2}\int_0^\infty(1+t)^{\beta-2}\omega^2dx
     +
      Ch \int_0^\infty  (1+t)^{\beta+\lambda-1} (\omega_x^2+\omega_t^2) dx
        +
         C\delta(1+t)^{\beta+\frac{3\lambda}{2}-\frac{5}{2}}.
\end{align}
Let $h=\frac{6}{\alpha}$, and  $T_1$ sufficiently large such that \begin{equation}\notag
\left\{
\begin{array}{l}
Ch(1+t)^{\lambda-1}\leq \frac{C_0}{4},\\[2mm]
Ch(1+t)^{\lambda-1}\leq \frac{1}{2}\left( \frac{\alpha h}{2}-1\right) ,\\[2mm]
\frac{\beta}{2}(1+t)^{\lambda-1}\leq \frac{1}{4},
\end{array}\right.
\end{equation}
 for $t\geq T_1$, then we have
\begin{equation}\notag
\begin{split}
&
\frac{\rm d}{{\rm d}t} \mathcal{H}_1(t) +\frac{\alpha(\lambda-\beta)}{4}\int_0^\infty(1+t)^{\beta-\lambda-1}  \omega^2dx
    +
    \frac{C_0}{4}\int_0^\infty(1+t)^{\beta}\omega_x^2dx
      +
        \frac{1}{2}\left( \frac{\alpha h}{2}-1\right) \int_0^\infty (1+t)^{\beta}\omega_t^2dx
         \\
        \leq
        &
          C\frac{\rm d}{{\rm d}t}\int_0^\infty (1+t)^{\beta+\lambda}\bigg(  \int_{\bar v}^{\omega_x+\bar v+\hat v} p(s)ds-p(\bar v)\omega_x-\frac{p'(\bar v)}{2}\omega_x^2 \bigg)dx
        +\frac{\beta}{2}\frac{\rm d}{{\rm d}t} \int_0^\infty(1+t)^{\beta-1}\omega^2dx\\
          &
          +
     \frac{\beta(1-\beta)}{2}\int_0^\infty(1+t)^{\beta-2}\omega^2dx
       +
         C\delta(1+t)^{\beta+\frac{3\lambda}{2}-\frac{5}{2}},
\end{split}
\end{equation}
where
\begin{equation}\notag
  \mathcal{H}_1(t)
    \sim
     (1+t)^{\beta-\lambda}\|\omega\|^2 +
       (1+t)^{\beta+\lambda}\|\omega_t\|^2
      +
       (1+t)^{\beta+\lambda}\|\omega_x\|^2
       .
\end{equation}
Therefore, we know
\begin{equation}\notag
\begin{split}
&
\frac{\rm d}{{\rm d}t}\mathcal{H}_1(t)
    +
    \frac{C_0}{4}\int_0^\infty(1+t)^{\beta}\omega_x^2dx
      +
        c\int_0^\infty  (1+t)^{\beta}\omega_t^2dx
        +c\int_0^\infty(1+t)^{\beta-\lambda-1}  \omega^2dx
         \\
        \leq
        &
          C\frac{\rm d}{{\rm d}t}\int_0^\infty (1+t)^{\beta+\lambda}\bigg(  \int_{\bar v}^{\omega_x+\bar v+\hat v} p(s)ds-p(\bar v)\omega_x-\frac{p'(\bar v)}{2}\omega_x^2 \bigg)dx+\frac{\beta}{2}\frac{\rm d}{{\rm d}t} \int_0^\infty(1+t)^{\beta-1}\omega^2dx\\
     &
    +
     C(1+t)^{\lambda-2}  \mathcal{H}_1(t)
        +
         C\delta(1+t)^{\beta+\frac{3\lambda}{2}-\frac{5}{2}}, \qquad \mbox{for any}\qquad  t\in [T_1, \infty).
         \end{split}
\end{equation}
Using  Gronwall’s inequality on $[T_1,t]$,  Proposition \ref{prop 3}, $\beta+\frac{3\lambda}{2}-\frac{5}{2}>-1 $  and $\frac{3}{5}< \lambda< 1$, one derives
\begin{equation}\label{3.182}
\begin{split}
&
 \mathcal{H}_1(t)
    +
    c\int_{T_1}^t \bigg[(1+s)^{\beta-\lambda-1}\|\omega\|^2+ (1+s)^{\beta}\big(\|\omega_x\|^2
      +
        \|\omega_t\|^2\big)\bigg]ds
         \\
        \leq
        &
          C\int_0^\infty (1+t)^{\beta+\lambda}\bigg(  \int_{\bar v}^{\omega_x+\bar v+\hat v} p(s)ds-p(\bar v)\omega_x-\frac{p'(\bar v)}{2}\omega_x^2 \bigg)dx
          +\frac{\beta }{2}\int_0^\infty(1+t)^{\beta-1}\omega^2dx\\
     &
        +\mathcal{H}_1(T_1)+C \delta(1+t)^{\beta+\frac{3\lambda}{2}-\frac{3}{2}}\\
       \leq
        &
          C\epsilon\int_0^\infty (1+t)^{\beta+\lambda}|\omega_x|^2dx+\frac{1 }{4}\int_0^\infty(1+t)^{\beta-\lambda}\omega^2dx
          +\mathcal{H}(T_1)+C\delta(1+t)^{\beta+\frac{3}{2}\lambda-\frac{3}{2}},
         \end{split}
\end{equation}
which together with Proposition \ref{prop 3}  deduce \eqref{3.1b} in view of the smallness of $\epsilon$ and $\beta+\frac{3\lambda}{2}-\frac{3}{2}>0$.

Hence we complete the proof of Lemma \ref{Lemma 1}.
\end{proof}

Furthermore, we can get the  better decay rate of  the functions $\omega_x$ and $\omega_t$ as follows:
\begin{lemma}\label{Lemma 2}
Under the assumptions of Theorems \ref{Thm 1}-\ref{Thm 1-1}, if $\epsilon, \delta$ are small, it holds that
\begin{equation}\label{3.21a}
\begin{split}
      &
      (1+t)^{\lambda+1}(\|\omega_x\|^2+\|\omega_t\|^2)
       +
       \int_0^t(1+s)\|\omega_t\|^2ds\\
\leq
 &
   C( \|\omega_0\|_{H^1}^2+\|z_0 \|^2+\delta),\qquad \mbox{for } \qquad 0\leq\lambda<\frac{3}{5},
\end{split}
\end{equation}
and for $\frac{3}{5}<\lambda<1$
\begin{equation}\label{3.21b}
\left\{
\begin{split}
&
       (1+t)^{\frac{5}{2}-\frac{3\lambda}{2}}(\|\omega_x\|^2+\|\omega_t\|^2)
       \leq
   C( \|\omega_0\|_{H^1}^2+\|z_0 \|^2+\delta),\\
&
       \int_0^t(1+s)^{\beta-\lambda+1}\|\omega_t\|^2ds\\
       &
\qquad \leq
   C(1+t)^{\beta+\frac{3\lambda}{2}-\frac{3}{2}}( \|\omega_0\|_{H^1}^2+\|z_0 \|^2+\delta),\quad  \mbox{for any }\quad  \frac{3}{2}-\frac{3\lambda}{2}<\beta<\lambda.
\end{split} \right.
\end{equation}
\end{lemma}
\begin{proof}
We only prove \eqref{3.21b}, and the proof of   \eqref{3.21a} is similar. In the case of $\frac{3}{5}<\lambda<1$, multiplying  \eqref{3.14} by $(1+t)^{1-\lambda}$ leads to
\begin{equation}\notag
\begin{split}
&
\frac{1}{2}\frac{\rm d}{{\rm d}t} \int_0^\infty
 \left[
   (1+t)^{\beta+1}\omega_t^2-(1+t)^{\beta+1}p'(\bar v)\omega_x^2
 \right]dx
  + \frac{\alpha}{2}\int_0^\infty (1+t)^{\beta-\lambda+1}\omega_t^2dx
  \\
\leq&
     \frac{\rm d}{{\rm d}t}\int_0^\infty (1+t)^{\beta+1}\bigg(  \int_{\bar v}^{\omega_x+\bar v+\hat v} p(s)ds-p(\bar v)\omega_x-\frac{p'(\bar v)}{2}\omega_x^2 \bigg)dx\\
     &
     +C\bigg|\int_0^\infty (1+t)^{\beta}\bigg(  \int_{\bar v}^{\omega_x+\bar v+\hat v} p(s)ds-p(\bar v)\omega_x-\frac{p'(\bar v)}{2}\omega_x^2 \bigg)dx\bigg|\\
     &
     +C\int_0^\infty\big( (1+t)^{\beta}\omega_t^2-(1+t)^{\beta}p'(\bar v)\omega_x^2\big)dx\\
   &
   +
    C \int_0^\infty  (1+t)^{\beta} (\omega_x^2+\omega_t^2) dx
        +
         C\delta(1+t)^{\beta+\frac{3\lambda}{2}-\frac{5}{2}}.
\end{split}
\end{equation}
Integrating the above inequality in $t$ over $(0,t)$, using Lemma \ref{Lemma 1} and $\beta+\frac{3\lambda}{2}>\frac{3}{2}$, we get
\begin{equation}\notag
\begin{split}
&
       (1+t)^{\beta+1}(\|\omega_x\|^2+\|\omega_t\|^2)+ \int_0^t(1+s)^{\beta-\lambda+1}\|\omega_t\|^2ds\\
      \leq
      &
   C(1+t)^{\beta+\frac{3\lambda}{2}-\frac{3}{2}}( \|\omega_0\|_{H^1}^2+\|z_0 \|^2+\delta).
   \end{split}
\end{equation}
This completes the proof of Lemma \ref{Lemma 2}.
\end{proof}

Similarly, we can derive decay rates on the higher
derivatives of the global solution $\omega(x,t)$.
\begin{lemma}\label{Lemma 3}
Under the assumptions of Theorems \ref{Thm 1}-\ref{Thm 1-1}, if $\epsilon, \delta$ are small, it holds that
\begin{equation}\notag
\begin{split}
      &
      (1+t)^{2\lambda+2}(\|\omega_{xx}\|^2+\|\omega_{xt}\|^2)
       +
       \int_0^t\big((1+s)^{2\lambda+1}\|\omega_{xx}\|^2+(1+s)^{\lambda+2}\|\omega_{xt}\|^2\big)ds\\
\leq
 &
   C( \|\omega_0\|_{H^2}^2+\|z_0 \|_{H^1}^2+\delta),\qquad \mbox{for any} \qquad 0\leq\lambda<\frac{3}{5},
\end{split}
\end{equation}
and for $\frac{3}{5}<\lambda<1$
\begin{equation}\notag
\left\{
\begin{split}
&
       (1+t)^{\frac{7}{2}-\frac{\lambda}{2}}(\|\omega_{xx}\|^2+\|\omega_{xt}\|^2)
       \leq
   C( \|\omega_0\|_{H^2}^2+\|z_0 \|_{H^1}^2+\delta),\\
&
       \int_0^t\big(
       (1+s)^{\beta+\lambda+1}\|\omega_{xx}\|^2
       +
       (1+s)^{\beta+2}\|\omega_{xt}\|^2
\big)ds\\
&\qquad \leq
   C(1+t)^{\beta+\frac{3\lambda}{2}-\frac{3}{2}}( \|\omega_0\|_{H^2}^2+\|z_0 \|_{H^1}^2+\delta), \quad  \mbox{for any }\quad  \frac{3}{2}-\frac{3\lambda}{2}<\beta<\lambda.
\end{split} \right.
\end{equation}
\end{lemma}

\begin{lemma}\label{Lemma 4}
Under the assumptions of Theorems \ref{Thm 1}-\ref{Thm 1-1},  if $\epsilon, \delta$ are small, it holds that
\begin{equation}\notag
\begin{split}
      &
      (1+t)^{3\lambda+3}(\|\omega_{xxx}\|^2+\|\omega_{xxt}\|^2)
       +
       \int_0^t\big((1+s)^{3\lambda+2}\|\omega_{xxx}\|^2+(1+s)^{2\lambda+3}\|\omega_{xxt}\|^2\big)ds\\
\leq
 &
   C( \|\omega_0\|_{H^3}^2+\|z_0 \|_{H^2}^2+\delta),   \qquad \mbox{for any} \qquad 0\leq\lambda<\frac{3}{5},
\end{split}
\end{equation}
and for $\frac{3}{5}<\lambda<1$
\begin{equation}\notag
\left\{
\begin{split}
&
       (1+t)^{\frac{9}{2}+\frac{\lambda}{2}}(\|\omega_{xxx}\|^2+\|\omega_{xxt}\|^2)
       \leq
   C( \|\omega_0\|_{H^3}^2+\|z_0 \|_{H^2}^2+\delta),\\
&
       \int_0^t\big(
       (1+s)^{\beta+2\lambda+2}\|\omega_{xxx}\|^2
       +
       (1+s)^{\beta+\lambda+3}\|\omega_{xxt}\|^2
       \big)ds\\
       &
\qquad \leq
   C(1+t)^{\beta+\frac{3\lambda}{2}-\frac{3}{2}}( \|\omega_0\|_{H^3}^2+\|z_0 \|_{H^2}^2+\delta),\quad  \mbox{for any }\quad  \frac{3}{2}-\frac{3\lambda}{2}<\beta<\lambda.
\end{split} \right.
\end{equation}
\end{lemma}
With Lemmas \ref{Lemma 1}-\ref{Lemma 4} in hand, by the Sobolev inequality and $0\leq \lambda<1$, it's easy to  know that
\begin{equation}\notag
\begin{split}
       \| \omega_x(\cdot,t)\|_{L^\infty}
       &\leq
   C(1+t)^{-\frac{3}{4}}( \|\omega_0\|_{H^3}+\|z_0 \|_{H^2}+\delta^\frac{1}{2})
\leq \frac{\epsilon}{2},
   \end{split}
\end{equation}
  \begin{equation}\notag
  \begin{split}
       \| \omega_{xt}(\cdot,t)\|_{L^\infty}&\leq
   C(1+t)^{-\frac{7}{4}}( \|\omega_0\|_{H^3}+\|z_0 \|_{H^2}+\delta^\frac{1}{2})
 \leq \frac{\epsilon}{2}(1+t)^{-1},
   \end{split}
\end{equation}
and
  \begin{equation}\notag
  \begin{split}
       \| \omega_{xx}(\cdot,t)\|_{L^\infty}&\leq
   C(1+t)^{-\frac{5}{4}}( \|\omega_0\|_{H^3}+\|z_0 \|_{H^2}+\delta^\frac{1}{2})
 \leq \frac{\epsilon}{2}(1+t)^{-\frac{\lambda+1}{2}}
   \end{split}
\end{equation}
provided $\|\omega_0\|_{H^3}+\|z_0 \|_{H^2}+\delta\ll 1$. Up to now, we thus close the {\it a priori} assumption \eqref{xian}  about $(\omega_x, \omega_{xt}, \omega_{xx})$ from  Lemmas \ref{Lemma 1}-\ref{Lemma 4}.

In next two lemmas we want to pay an attention to $z=\omega_t$ which has the improved decay rates.
\begin{lemma}\label{Lemma 5}
Under the assumptions of Theorems \ref{Thm 1}-\ref{Thm 1-1}, if $\epsilon, \delta$ are small, it holds that
\begin{equation}\notag
\begin{split}
      &
       (1+t)^{2}\|z\|^2+(1+t)^{\lambda+3}(\|z_x\|^2+\|z_t\|^2)
       +
       \int_0^t((1+s)^{\lambda+2}\|z_x\|^2+(1+s)^{3}\|z_t\|^2)ds\\
\leq
 &
   C( \|\omega_0\|_{H^2}^2+\|z_0 \|_{H^1}^2+\delta),  \qquad \mbox{for any} \qquad 0\leq\lambda<\frac{3}{5},
\end{split}
\end{equation}
and for $\frac{3}{5}<\lambda<1$
\begin{equation}\notag
\left\{
\begin{split}
&
        (1+t)^{\frac{7}{2}-\frac{5\lambda}{2}}\|z\|^2
        +
        (1+t)^{\frac{9}{2}-\frac{3\lambda}{2}}(\|z_{x}\|^2+\|z_{t}\|^2)
       \leq
   C( \|V_0\|_{H^2}^2+\|z_0 \|_{H^1}^2+\delta),\\
&
       \int_0^t\big(
       (1+s)^{\beta+2}\|z_{x}\|^2
       +
       (1+s)^{\beta-\lambda+3}\|z_{t}\|^2
       \big)ds\\
       &
\qquad \leq
   C(1+t)^{\beta+\frac{3\lambda}{2}-\frac{3}{2}}( \|\omega_0\|_{H^2}^2+\|z_0 \|_{H^1}^2+\delta),\quad  \mbox{for any }\quad  \frac{3}{2}-\frac{3\lambda}{2}<\beta<\lambda.
\end{split} \right.
\end{equation}
\end{lemma}

Similar to the proof of Lemma \ref{Lemma 5}, we can get the
following estimates.
\begin{lemma}\label{Lemma 6}
Under the assumptions of Theorems \ref{Thm 1}-\ref{Thm 1-1},  if $\epsilon, \delta$ are small, it holds that
\begin{equation}\notag
\begin{split}
      &
      (1+t)^{2\lambda+4}(\|z_{xx}\|^2+\|z_{xt}\|^2)
       +
       \int_0^t\big((1+s)^{2\lambda+3}\|z_{xx}\|^2+(1+s)^{\lambda+4}\|z_{xt}\|^2\big)ds\\
\leq
 &
   C( \|\omega_0\|_{H^3}^2+\|z_0 \|_{H^2}^2+\delta),  \qquad \mbox{for any} \qquad 0\leq\lambda<\frac{3}{5},
\end{split}
\end{equation}
and for $\frac{3}{5}<\lambda<1$
\begin{equation}\notag
\left\{
\begin{split}
&
        (1+t)^{\frac{11}{2}-\frac{\lambda}{2}}(\|z_{xx}\|^2+\|z_{xt}\|^2)
       \leq
   C( \|\omega_0\|_{H^3}^2+\|z_0 \|_{H^2}^2+\delta),\\
&
       \int_0^t\big(
       (1+s)^{\beta+\lambda+3}\|z_{xx}\|^2
       +
       (1+s)^{\beta+4}\|z_{xt}\|^2
\big)ds\\
       &
\qquad
      \leq
   C(1+t)^{\beta+\frac{3\lambda}{2}-\frac{3}{2}}( \|\omega_0\|_{H^3}^2+\|z_0 \|_{H^2}^2+\delta),\quad  \mbox{for any }\quad  \frac{3}{2}-\frac{3\lambda}{2}<\beta<\lambda.
\end{split} \right.
\end{equation}
\end{lemma}
Recalling Lemmas \ref{Lemma 1}-\ref{Lemma 6}, we complete the proofs of Theorems \ref{Thm 1}-\ref{Thm 1-0}. The proof of Theorem  \ref{Thm 1-1} is similar.

\vspace{6mm}
}

\section{The case of Neumann boundary condition}\label{S4}
\numberwithin{equation}{section}

In this section, we consider the problem \eqref{1.3}-\eqref{1.4} with the  null-Neumann boundary condition \eqref{bound2}, i.e.
\begin{equation}\label{t3.1}
\left\{\begin{array}{l}
\partial_t  v
  -
   \partial_x u=0,\\[2mm]
\partial_t u
   +
      \partial_x  p(v)
     =\displaystyle
     -\frac{\alpha}{(1+t)^\lambda} u,\qquad (x,t)\in \mathbb{R}^+\times\mathbb{R}^+,\\
     u_x(0,t)=0,
 \end{array}
        \right.
\end{equation}
and the initial data is
\begin{equation}\label{t3.2}
  (v,u)\mid_{t=0}=(v_0,u_0)(x)\rightarrow (v_+,u_+),
   \quad
    \mbox{as}
     \quad
     x\rightarrow  +\infty
      \quad
        \mbox{and}
         \quad
           v_+>0.
\end{equation}
From $(\ref{t3.1})_1$ and the boundary condition $(\ref{t3.1})_3$, we have $v(0,t)=v_0(0)$ for any $t>0$.

\subsection{The case of $v_0(0)\neq v_+$}
\numberwithin{equation}{subsection}

To construct the   diffusion waves $(\bar v,\bar u)(x,t)$ to the corresponding boundary condition, from \cite{Cui-Yin-Zhang-Zhu2016,Hsiao-Liu1992,Nishihara-Yang1999}, it is known that for any  two constants $v_\pm>0$ there exists a unique self-similar solution $\tau(x,t)=\phi(\frac{x}{(1+t)^{\frac{\lambda+1}{2}}})$ satisfying
\begin{equation}\notag
\left\{\begin{array}{l}
 \displaystyle\frac{ \tau_t}{(1+t)^\lambda}
 -
  \frac{1}{\alpha} p(\tau)_{xx} =0, \qquad (x,t)\in \mathbb{R}\times\mathbb{R}^+, \\[2mm]
     \tau( \pm \infty,t)=v_\pm.
 \end{array}
        \right.
\end{equation}
Therefore, for any constant $v_0(0)>0$ between $v_-$ and $v_+$, there exists a unique $\bar v(x,t)$ in the form of $\phi(\frac{x}{(1+t)^{\frac{\lambda+1}{2}}})|_{t\geq 0}$ satisfying
\begin{equation}\notag
\left\{\begin{array}{l}
 \displaystyle\frac{ \bar v_t}{(1+t)^\lambda}
 -
  \frac{1}{\alpha}  p(\bar v)_{xx}=0, \qquad (x,t)\in \mathbb{R}^+\times\mathbb{R}^+,\\[2mm]
   \bar v (0,t)=v_0(0), \qquad \bar v (\infty,t)=v_+.
 \end{array}
        \right.
\end{equation}
We set
\begin{equation}\notag
     \bar{u}(x,t)=\displaystyle-\frac{(1+t)^\lambda}{\alpha}p(\bar v)_x,
\end{equation}
so that $\bar u_x(0,t)=\bar v_t(0,t)=x\phi'(\frac{x}{(1+t)^{\frac{\lambda+1}{2}}})
(-\frac{\lambda+1}{2})(1+t)^{-\frac{\lambda+3}{2}}\mid_{x=0}=0$.
 Thus, $(\bar v,\bar u)(x,t)$ called the nonlinear diffusion wave, satisfies
 \begin{equation}\label{t3.5}
\left\{\begin{array}{l}
  \bar{v}_t
  -
     \bar{ u}_x=0,\\[2mm]
      p(\bar v)_x
     =\displaystyle
     -\frac{\alpha}{(1+t)^\lambda} \bar{u},\qquad \qquad (x,t)\in \mathbb{R}^+\times\mathbb{R}^+,\\[2mm]
     (\bar v,\bar u_x)(0,t)=(v_0(0),0),\qquad (\bar v, \bar{u})(  \infty,t)=(v_+,0),\\[2mm]
      (\bar v,\bar u)\mid_{t=0}=(\bar v_0, \bar u_0)(x)\rightarrow (v_+,0),
   \quad
    \mbox{as}
     \quad
     x\rightarrow  +\infty.
 \end{array}
        \right.
\end{equation}
According to the idea in \cite{Nishihara-Yang1999}, similar to that in the Dirichlet boundary problem, the correction functions is defined by
\begin{equation}\label{t3.6}
\hat{v}(x,t)=-(u_0(0)-u_+) m_0(x)B(t),
\end{equation}
and
\begin{equation}\label{t3.7}
\hat{u}(x,t)=\bigg[u_++(u_0(0)-u_+)\int_{x}^\infty m_0(y)dy\bigg]\beta(t),
\end{equation}
where $\beta(t)$, $B(t)$ have been defined in Section 2.1, and  $m_0(x)$ is a smooth function with compact support such that
\begin{equation}\notag
\int_{0}^\infty m_0(x)dx=1, \qquad \mbox{supp}\  m_0\subset \mathbb{R}^+.
\end{equation}
Therefore, $(\hat v,\hat u)(x,t)$ satisfies
\begin{equation}\label{t3.8}
\left\{\begin{array}{l}
\partial_t  \hat{v}
  -
   \partial_x \hat{ u}=0,\\[2mm]
      \partial_t\hat{u}
     =\displaystyle
     -\frac{\alpha}{(1+t)^\lambda}\hat{u},
     \qquad   (x,t)\in \mathbb{R}^+\times\mathbb{R}^+,  \\[2mm]
     (\hat v, \hat u_x)(0,t)=(0,0), \qquad \hat u(0,t)=u_0(0)\beta(t),\\[2mm]
     (\hat v, \hat u)(\infty,t)=(0,u_+\beta(t)).
 \end{array}
        \right.
\end{equation}
Combining \eqref{t3.1}, \eqref{t3.5} and \eqref{t3.8}, it is easy to see that
\begin{equation}\label{t3.9}
\left\{\begin{array}{l}
\partial_t  (v-\bar{v}-\hat{v})
  -
   \partial_x (u-\bar{ u}-\hat{u})=0,
   \qquad (x,t)\in \mathbb{R}^+\times\mathbb{R}^+, \\[2mm]
      \partial_t (u-\bar{u}-\hat{u})+\partial_x  (p(v)-p(\bar{v}))
         + \partial_t \bar u
           +\displaystyle
              \frac{\alpha}{(1+t)^\lambda} (u-\bar{u}-\hat{u})=0,\\
          ( u-\bar{u}-\hat{u})_x(0,t)=0.
 \end{array}
        \right.
\end{equation}
Hence defining the perturbation by
\begin{equation}\label{t3.10}
\omega(x,t)=-\int_x^{\infty} (v(y,t)-\bar v(y,t)-\hat{v}(y,t))dy,
\end{equation}
\begin{equation}\label{t3.11}
z(x,t)= u(x,t)-\bar u(x,t)-\hat{u}(x,t).
\end{equation}
It follows from \eqref{t3.9} that the reformulated problem is
\begin{equation}\label{t3.12}
\left\{\begin{array}{l}
  \omega_t
  -
   z=0, \qquad  (x,t)\in \mathbb{R}^+\times\mathbb{R}^+ ,\\[2mm]
       z_t+ (p(\omega_x+\bar v+\hat v)-p(\bar{v}))_x
           +\displaystyle
              \frac{\alpha}{(1+t)^\lambda} z=-\bar u_t,\\[2mm]
              \omega_x(0,t)=0, \qquad z_x(0,t)=0,
 \end{array}
        \right.
\end{equation}
with initial data
\begin{equation}\label{t3.13}
  (\omega,z)\mid_{t=0}=(\omega_0,z_0)(x),
\end{equation}
where
\begin{equation}\notag
\left\{\begin{array}{l}
  \ds \omega_0(x)
    =-\int_x^{\infty}  (v_0(y)-\bar v(y,0)-\hat{v}(y,0))dy,\\
   \ds z_0(x)
       =u_0(x)-\bar u(x,0)-\hat{u}(x,0).
\end{array}
\right.
\end{equation}
Rewrite  \eqref{t3.12} and \eqref{t3.13} as
\begin{equation}\label{t3.15}
\left\{\begin{array}{l}
      \omega_{tt}+ (p'(\bar{v})\omega_x)_x
           +\displaystyle
              \frac{\alpha}{(1+t)^\lambda} \omega_t=F_2, \qquad (x,t)\in \mathbb{R}^+\times\mathbb{R}^+,\\
               \omega_x(0,t)=0,
 \end{array}
 \right.
\end{equation}
with initial data
\begin{equation}\label{t3.16}
  (\omega,\omega_t)\mid_{t=0}=(\omega_0,z_0)(x),
\end{equation}
where
\begin{equation}\label{t3.17}
\begin{split}
  F_2
   =
    \frac{1}{\alpha}(1+t)^\lambda p(\bar v)_{xt}
      +
       \frac{\lambda}{\alpha} { (1+t)^{\lambda-1} } { p(\bar v)_x}
          -
            (p(\omega_x+\bar v+\hat v)-p(\bar{v})-p'(\bar v)\omega_x)_x.
 \end{split}
\end{equation}

Note that, from the boundary condition, we have
\begin{equation}\notag
\omega_{x}(0,t)=\omega_{tx}(0,t)=\omega_{ttx}(0,t)= \big(p(\omega_x+\bar v+\hat v)-p(\bar v)\big) \mid_{x=0}
 =0,
 \qquad\quad \qquad \mbox{etc}.
\end{equation}

The nonlinear diffusion wave $\bar v(x,t)$ defined in \eqref{t3.5} has the same behavior as in \cite{Cui-Yin-Zhang-Zhu2016}.  It's worth noting that their convergence rates are different from Proposition \ref{Prop a1}. Therefore, our main result about Neumann boundary condition in this section  is different from  the case of Dirichlet boundary condition (see Theorems \ref{Thm 1}-\ref{Thm 1-1}).

%%%%%%%%%%%%%%%%%%%%%%%%%%%%%%%%%%%%%%%%%%%%%%%%%%%%%%%%%%%%%%%%%%%%%%%%%%%%%%%%%%%%%%%%%%%%%%
\begin{theorem}\label{Thm t2}{\bf(The case of $v_0(0)\neq v_+$ and $0\leq \lambda <\frac{1}{7}$)}
For any $\alpha>0$ and $v_0-v_+\in L^1(\mathbb{R}^+)$, assume that both $\delta_1=|v_+-v_0(0)|+|u_+-u_0(0)|$ and
 $ \|\omega_0\|_{H^3(\mathbb{R}^+)}+\|z_0\|_{H^2(\mathbb{R}^+)}$ are sufficiently small. Then,
there exists a unique time-global solution of the problem \eqref{t3.15}-\eqref{t3.16} satisfying
\begin{equation}\notag
  \omega\in C^{k}((0,\infty),H^{3-k}(\mathbb{R}^+)),\qquad k=0,1,2,3,
\end{equation}
\begin{equation}\notag
  \omega_t\in C^{k}((0,\infty),H^{2-k}(\mathbb{R}^+)),\qquad k=0,1,2,
\end{equation}
furthermore, we have
\begin{equation}\notag
\begin{split}
&
      \sum_{k=0}^3(1+t)^{(\lambda +1)k }\|\partial_x^k\omega(\cdot,t)\|_{L^2}^2
      +
         \sum_{k=0}^2(1+t)^{(\lambda+1)k +2} \|\partial_x^k\omega_t(\cdot,t)\|_{L^2}^2
          \\
         &
            +
                \int_0^t \bigg[ \sum_{j=1}^3(1+s)^{(\lambda +1)j-1}\|\partial_x^j\omega(\cdot,s)\|_{L^2}^2
                 +
                 \sum_{j=0}^2(1+s)^{(\lambda +1)j+1} \|\partial_x^j\omega_t(\cdot,s)\|_{L^2}^2\bigg]ds\\
                    \leq
                       &
                          C(\|\omega_0\|_{H^3}^2+\|z_0\|_{H^2}^2+\delta_1).
\end{split}
\end{equation}
\end{theorem}

\begin{theorem}\label{Thm t3}{\bf(The case of $v_0(0)\neq v_+$ and $ \frac{1}{7}<\lambda <1$)}
For any $\alpha>0$ and $v_0-v_+\in L^1(\mathbb{R}^+)$, assume that both $\delta_1=|v_+-v_0(0)|+|u_+-u_0(0)|$ and
 $ \|\omega_0\|_{H^3(\mathbb{R}^+)}+\|z_0\|_{H^2(\mathbb{R}^+)}$ are sufficiently small. Then,
there exists a unique time-global solution of the problem \eqref{t3.15}-\eqref{t3.16} satisfying
\begin{equation}\notag
  \omega\in C^{k}((0,\infty),H^{3-k}(\mathbb{R}^+)),\qquad k=0,1,2,3,
\end{equation}
\begin{equation}\notag
  \omega_t\in C^{k}((0,\infty),H^{2-k}(\mathbb{R}^+)),\qquad k=0,1,2,
\end{equation}
furthermore, we have
\begin{equation}\notag
\begin{split}
     &
      \sum_{k=0}^3(1+t)^{(\lambda +1)k +\frac{1}{2}-\frac{7\lambda}{2}}
      \|\partial_x^k\omega(\cdot,t)\|_{L^2}^2
      +
         \sum_{k=0}^2(1+t)^{(\lambda+1)k +\frac{5}{2}-\frac{7\lambda}{2}} \|\partial_x^k\omega_t(\cdot,t)\|_{L^2}^2\\
                 \leq
                       &
                          C(\|\omega_0\|_{H^3}^2+\|z_0\|_{H^2}^2+\delta),
\end{split}
\end{equation}
and for any $\beta\in(\frac{1}{2}-\frac{5\lambda}{2},\lambda),$ we
have
\begin{equation}\notag
\begin{split}
         &
                \int_0^t \bigg[ \sum_{j=0}^3(1+s)^{(\lambda +1)(j-1)+\beta}\|\partial_x^j\omega(\cdot,s)\|_{L^2}^2
                 +
                 \sum_{j=0}^2(1+s)^{(\lambda +1)j+\beta-\lambda+1} \|\partial_x^j\omega_t(\cdot,s)\|_{L^2}^2\bigg]ds\\
                    \leq
                       &
                       C(1+t)^{\beta+\frac{5\lambda}{2}-\frac{1}{2}}
                           (\|\omega_0\|_{H^3}^2+\|z_0\|_{H^2}^2+\delta_1).
\end{split}
\end{equation}
\end{theorem}

\begin{theorem}\label{Thm t4}{\bf(The case of $v_0(0)\neq v_+$ and $\lambda =\frac{1}{7}$)}
For any $\alpha>0$ and $v_0-v_+\in L^1(\mathbb{R}^+)$, assume that both $\delta_1=|v_+-v_0(0)|+|u_+-u_0(0)|$ and
 $ \|\omega_0\|_{H^3(\mathbb{R}^+)}+\|z_0\|_{H^2(\mathbb{R}^+)}$ are sufficiently small. Then,
there exists a unique time-global solution of the problem \eqref{t3.15}-\eqref{t3.16} satisfying
\begin{equation}\notag
  \omega\in C^{k}((0,\infty),H^{3-k}(\mathbb{R}^+)),\qquad k=0,1,2,3,
\end{equation}
\begin{equation}\notag
  \omega_t\in C^{k}((0,\infty),H^{2-k}(\mathbb{R}^+)),\qquad k=0,1,2,
\end{equation}
furthermore, we have for any sufficiently small $\varepsilon>0$
\begin{equation}\notag
\begin{split}
     &
      \sum_{k=0}^3(1+t)^{\frac{8k}{7} }
      \|\partial_x^k\omega(\cdot,t)\|_{L^2}^2
      +
         \sum_{k=0}^2(1+t)^{\frac{8k}{7}+2} \|\partial_x^k\omega_t(\cdot,t)\|_{L^2}^2
          \\
         &
            +
                \int_0^t \bigg[ \sum_{j=1}^3(1+s)^{\frac{8j}{7}-1}\|\partial_x^j\omega(\cdot,s)\|_{L^2}^2
                 +
                 \sum_{j=0}^2(1+s)^{\frac{8j}{7}+1} \|\partial_x^j\omega_t(\cdot,s)\|_{L^2}^2\bigg]ds
         \\
                 \leq
                       &
                          C(1+t)^\varepsilon(\|\omega_0\|_{H^3}^2+\|z_0\|_{H^2}^2+\delta_1).
\end{split}
\end{equation}
\end{theorem}
\vspace{6mm}
Theorems \ref{Thm t2}-\ref{Thm t4} can be proved by using the similar method as Theorems \ref{Thm 1}-\ref{Thm 1-1}  and the details of the proofs are omitted here.

\subsection{The case of $v_0(0)= v_+$}

Let \begin{equation}\label{n4.1}
(\bar v,\bar u)(x,t)\equiv (v_+,0).
\end{equation}
Similar to the case as above, the correction functions are defined by
\begin{equation}\notag
\hat{v}(x,t)=-(u_0(0)-u_+) m_0(x)B(t),
\end{equation}
and
\begin{equation}\notag
\hat{u}(x,t)=\bigg[u_++(u_0(0)-u_+)\int_{x}^\infty m_0(y)dy\bigg]\beta(t),
\end{equation}
where $m_0(x)$ is a smooth function with compact support such that
\begin{equation}\notag
\int_{0}^\infty m_0(x)dx=1, \qquad \mbox{supp}\  m_0\subset \mathbb{R}^+.
\end{equation}
Therefore, $(\hat v,\hat u)(x,t)$ satisfies
\begin{equation}\label{n4.4}
\left\{\begin{array}{l}
\partial_t  \hat{v}
  -
   \partial_x \hat{ u}=0,\\[2mm]
      \partial_t\hat{u}
     =\displaystyle
     -\frac{\alpha}{(1+t)^\lambda}\hat{u},\\[2mm]
     (\hat v, \hat u_x)(0,t)=(0,0), \qquad \hat u(0,t)=u_+\beta(t),\\[2mm]
     (\hat v, \hat u)(\infty,t)=(0,u_+\beta(t)).
 \end{array}
        \right.
\end{equation}
It follows from \eqref{t3.1} and \eqref{n4.4} that
\begin{equation}\label{n4.5}
\left\{\begin{array}{l}
\partial_t  (v-v_+-\hat{v})
  -
   \partial_x (u- \hat{u})=0,\\[2mm]
      \partial_t (u- \hat{u})+\partial_x  (p(v)-p( {v}_+))
           +\displaystyle
              \frac{\alpha}{(1+t)^\lambda} (u-\hat{u})=0.
 \end{array}
        \right.
\end{equation}
The definition of
\begin{equation}\notag%\label{n4.7}
\omega(x,t)=-\int_x^{\infty} (v(y,t)- v_+-\hat{v}(y,t))dy,
\end{equation}
and
\begin{equation}\notag%\label{n4.8}
z(x,t)= u(x,t)-\hat{u}(x,t),
\end{equation}
give the reformulated problem
\begin{equation}\label{n4.9}
\left\{\begin{array}{l}
  \omega_t
  -
   z=0,\\[2mm]
       z_t+ (p(\omega_x+v_++\hat v)-p(\bar{v}))_x
           +\displaystyle
              \frac{\alpha}{(1+t)^\lambda} z=-\bar u_t,\\[2mm]
              \omega_x(0,t)=0, \qquad z_x(0,t)=0,
 \end{array}
        \right.
\end{equation}
with initial data
\begin{equation}\label{n4.10}
  (\omega,z)\mid_{t=0}=(\omega_0,z_0)(x),
\end{equation}
where
\begin{equation}\notag
\left\{\begin{array}{l}
  \ds \omega_0(x)
    =-\int_x^{\infty}  (v_0(y)-v_+-\hat{v}(y,0))dy,\\
   \ds z_0(x)
       =u_0(x)-\hat{u}(x,0).
\end{array}
\right.
\end{equation}
Rewrite  \eqref{n4.9} and \eqref{n4.10} as
\begin{equation}\label{n4.12}
      \omega_{tt}+ (p'(\bar{v})\omega_x)_x
           +\displaystyle
              \frac{\alpha}{(1+t)^\lambda} \omega_t=F_3,
\end{equation}
with initial data
\begin{equation}\label{n4.13}
  (\omega,\omega_t)\mid_{t=0}=(\omega_0,z_0)(x),
\end{equation}
where
\begin{equation}\label{n4.14}
\begin{split}
  F_3
   =  -
            (p(\omega_x+v_++\hat v)-p(\bar{v})-p'(v_+)\omega_x)_x.
 \end{split}
\end{equation}

%%%%%%%%%%%%%%%%%%%%%%%%%%%%%%%%%%%%%%%%%%%%%%%%%%%%%%%%%%%%%%%%%%%%%%%%%%%%%%%%%%%%%%%%%%%%%%
We note that the nonlinear term  $F_3$ in \eqref{n4.14} does not include the bad term  $\frac{1}{\alpha}(1+t)^\lambda p(\bar v)_{xt}
      $, $\frac{\lambda}{\alpha} { (1+t)^{\lambda-1} } { p(\bar v)_x}$ in \eqref{t3.17}. Thus, applying the same method as Theorems \ref{Thm t2}-\ref{Thm t4}, we can obtain
 our final  result.
\begin{theorem}\label{Thm n3}{\bf(The case of $v_0(0)= v_+$)}
For any $\alpha>0$ and $v_0-v_+\in L^1(\mathbb{R}^+)$,  assume that both $\delta_2=|u_+-u_0(0)|$ and $\|\omega_0\|_{H^3}+\|z_0\|_{H^2}$ is sufficiently small.  Then, for any $0\leq\lambda<1$,
there exists a unique time-global solution of the problem  \eqref{n4.12}-\eqref{n4.13} satisfying
\begin{equation}\notag
  \omega\in C^{k}((0,\infty),H^{3-k}(\mathbb{R}^+)),\qquad k=0,1,2,3,
\end{equation}
\begin{equation}\notag
  \omega_t\in C^{k}((0,\infty),H^{2-k}(\mathbb{R}^+)),\qquad k=0,1,2,
\end{equation}
furthermore, we have
\begin{equation}\notag
\begin{split}
&
      \sum_{k=0}^3(1+t)^{(\lambda +1)k }\|\partial_x^k\omega(\cdot,t)\|_{L^2}^2
      +
         \sum_{k=0}^2(1+t)^{(\lambda+1)k +2} \|\partial_x^k\omega_t(\cdot,t)\|_{L^2}^2
          \\
         &
            +
                \int_0^t \bigg[ \sum_{j=1}^3(1+s)^{(\lambda +1)j-1}\|\partial_x^j\omega(\cdot,s)\|_{L^2}^2
                 +
                 \sum_{j=0}^2(1+s)^{(\lambda +1)j+1} \|\partial_x^j\omega_t(\cdot,s)\|_{L^2}^2\bigg]ds\\
                    \leq
                       &
                          C(\|\omega_0\|_{H^3}^2+\|z_0\|_{H^2}^2+\delta_2).
\end{split}
\end{equation}
\end{theorem}
\begin{remark}\label{remark 6}
 It should be noted that there is no    cut-off point of the convergence rate in this case.
\end{remark}

\vspace{6mm}

\noindent {\bf Acknowledgements:} The research was
supported by the National Natural Science Foundation of China
\#11331005, \#11771150, \#11601164 and \#11601165.

\bigbreak

%
%
%
%\vspace{8mm}
{\small

\bibliographystyle{plain}

}

\end{document}